\documentclass[10pt,a4paper,twoside,reqno]{amsart}

\usepackage[a4paper,
left=1in,right=1in,top=1in,bottom=1.2in,%
footskip=.25in]{geometry}

\usepackage{bm}
\usepackage{latexsym, amsmath, amstext, amssymb, amsfonts, amscd, bm, array, multirow, amsbsy, mathrsfs}
\usepackage{amsthm}
\usepackage{t1enc}
\usepackage[mathscr]{eucal}
\usepackage{indentfirst}
\usepackage{pb-diagram}
\usepackage{graphicx}
\usepackage{fancyhdr}
\usepackage{fancybox}
\usepackage{enumerate}
\usepackage{color}
\usepackage{tikz-cd}
\usetikzlibrary{arrows}
\usepackage[all]{xy}
\usepackage{hyperref}
\usepackage{tikz}
\usepackage{xparse}
\hypersetup{colorlinks=false,pdfborderstyle={/S/U/W 0}}
\usetikzlibrary{matrix}
\usepackage{upgreek}

\newcommand{\arxiv}[1]{{\tt
    \href{http://www.arXiv.org/abs/#1}{arXiv:#1}}}

\usepackage{url}
\usepackage[sort&compress,comma]{natbib}
\bibpunct{[}{]}{,}{n}{}{,}
\theoremstyle{plain}
\newtheorem{thm}{Theorem}[section]

\newtheorem{prop}[thm]{Proposition}

\newtheorem{lemma}[thm]{Lemma}
\newtheorem{cor}[thm]{Corollary}

\theoremstyle{definition}
\newtheorem{definition}[thm]{Definition}

\theoremstyle{remark}
\newtheorem{remark}[thm]{Remark}
\newtheorem{example}[thm]{Example}

\newtheorem*{ack}{Acknowledgements}

\newcommand{\End}{\mathrm{End}}

\DeclareFontFamily{U}{rsf}{}
\DeclareFontShape{U}{rsf}{m}{n}{<5> <6> rsfs5 <7> <8> <9> rsfs7 <10-> rsfs10}{}
\DeclareMathAlphabet\Scr{U}{rsf}{m}{n}

\def\R{\mathbb{R}}

\def\dd{\mathrm{d}}

\def\cQ{\mathcal{Q}}

\def\Id{\mathrm{Id}}

\def\fra{\mathfrak{a}}
\def\frq{\mathfrak{q}}
\def\frf{\mathfrak{f}}

\newcommand{\be}{\begin{equation*}}
\newcommand{\ee}{\end{equation*}}
\newcommand{\ben}{\begin{equation}}
\newcommand{\een}{\end{equation}}
\newcommand{\beqa}{\begin{eqnarray*}}
	\newcommand{\eeqa}{\end{eqnarray*}}
\newcommand{\beqan}{\begin{eqnarray}}
\newcommand{\eeqan}{\end{eqnarray}}

\newcommand{\Tr}{\mathrm{Tr}}

\def\cR{{\mathcal R}}

\def\cC{{\mathcal C}}

\def\cB{\Scr B}

\def\cH{\mathcal{H}}
\def\cK
{\mathcal{K}}
\newcommand{\Sol}{\mathrm{Sol}}
\newcommand{\Conf}{\mathrm{Conf}}

\def\Cl{\mathrm{Cl}}

\def\Spin{\mathrm{Spin}}

\def\Spin{\mathrm{Spin}}

\def\cD{\mathcal{D}}

\def\cA{\mathcal{A}}
\def\cE{\mathcal{E}}

\def\cP{\mathcal{P}}

\def\cP{\mathcal{P}}
\def\cF{\mathcal{F}}
\def\cC{\mathcal{C}}

\def\G_2{\mathrm{G_2}}

\def\cS{\mathcal{S}}

\def\cX{\mathcal{X}}

\def\G{\mathrm{G}}

\def\R{\mathbb{R}}

\def\cL{\mathcal{L}}
\def\cH{\mathcal{H}}
\def\dd{\mathrm{d}}

\def\frc{\mathfrak{F}}
\def\frc{\mathfrak{c}}

\def\Ric{\mathrm{Ric}}

\usepackage{enumitem}
\setlist[itemize]{leftmargin=*}

\newcolumntype{P}[1]{>{\centering\arraybackslash}p{#1}}

\usepackage[math]{anttor}
\usepackage[T1]{fontenc}

\begin{document}

\title[Differential spinors and Kundt three-manifolds with skew-torsion]{Differential spinors and Kundt three-manifolds with skew-torsion}

\author[C. S. Shahbazi]{C. S. Shahbazi} \address{Departamento de Matem\'aticas, Universidad UNED - Madrid, Reino de Espa\~na}
\email{cshahbazi@mat.uned.es} 
\address{Fakult\"at f\"ur Mathematik, Universit\"at Hamburg, Bundesrepublik Deutschland.}
\email{carlos.shahbazi@uni-hamburg.de}

\begin{abstract}
We develop the theory of spinorial polyforms associated with bundles of irreducible Clifford modules of non-simple real type, obtaining a precise characterization of the square of an irreducible real spinor in signature $(p-q)\equiv_8 1$ as a polyform belonging to a semi-algebraic real set. We use this formalism to study differential spinors on Lorentzian three-manifolds, proving that in this dimension and signature, every differential spinor is equivalent to an isotropic line preserved in a given direction by a metric connection with prescribed torsion. We apply this result to investigate Lorentzian three-manifolds equipped with a skew-torsion parallel spinor, namely a spinor parallel with respect to a metric connection with totally skew-symmetric torsion. We obtain several structural results about this class of Lorentzian three-manifolds, which are necessarily Kundt, and in the compact case, we obtain an explicit differential condition that guarantees geodesic completeness. We further elaborate on these results to study the supersymmetric solutions of three-dimensional NS-NS supergravity, which involve skew-torsion parallel spinors whose torsion is given by the curvature of a curving on an abelian bundle gerbe. In particular, we obtain a correspondence between NS-NS supersymmetric solutions and null coframes satisfying an explicit exterior differential system that we solve locally. 
\end{abstract}
 
\maketitle

\setcounter{tocdepth}{1} 
\tableofcontents


\section{Introduction}



\subsection*{Background and context}


The main purpose of this article is to introduce and explore the notion of differential spinor, which we consider here in signature $(p-q)\equiv_8 1$, as a unifying notion naturally containing as particular cases the different types of special parallel spinors that have been considered in the literature, as well as some that are yet to appear or be discovered. This includes parallel spinors, Killing spinors, Codazzi spinors, Cauchy spinors, skew-torsion parallel spinors, skew-Killing spinors or generalized Killing spinors, just to name a few, see for instance \cite{BarGM,BaumIII,Conti,FlamencourtMoroianu,Galaev,Kath,Ikemakhen,IkemakhenII,Leitner,MoroianuSemm,MoroianuSemmI,MoroianuSemmII} as well as their references and citations. 

Let $(M,g)$ be an oriented pseudo-Riemannian manifold and let $S$ a bundle of Clifford modules over the bundle of Clifford algebras of $(M,g)$. Given a connection $\cD$ on $S$ a differential spinor is a section $\varepsilon\in \Gamma(S)$ such that $\cD\varepsilon =0$. In this generality, the notion of differential spinor is in principle too general to yield explicit results, mainly because the connection $\cD$ is abstractly defined and is not tied \emph{a priori} to the geometry of $(M,g)$. In signatures $(p-q)\equiv_8 0, 1, 2$ however, the results of \cite{Lazaroiu:2016vov,LS2018} imply that a bundle of irreducible Clifford modules exists if and only if $(M,g)$ is spin, in which case there exists a spin structure $Q_g$ on $(M,g)$ such that $S$ is associated to $Q_g$ via the standard \emph{spinorial} representation of the spin group. Hence, in this case the Levi-Civita connection $\nabla^g$ lifts to $S$ and equation $\cD\varepsilon = 0$ reduces to:
\begin{equation}
\label{eq:differentialspinor}
\nabla^g\varepsilon = \cA(\varepsilon)
\end{equation} 

\noindent
where $\cA\in \Omega^1(\End(S))$ is a uniquely determined one-form on $M$ with values in the endomorphism bundle $\End(S)$. Equation \eqref{eq:differentialspinor} defines the notion of differential spinor that we will consider in this note and that represents the most general \emph{parallelicity condition} that one can impose on an irreducible real spinor in signatures $(p-q)\equiv_8 0, 1, 2$. To the best of our knowledge it does not seem to have been studied in this generality in the literature, especially not in Lorentzian signature. 

Our main reason for introducing the notion of differential spinor is to develop a general spinorial geometric framework that can be applied to the study of the differential conditions satisfied by supersymmetric solutions and configurations in supergravity and string theory, see \cite{Gran:2018ijr,Ortin,Tomasiello} and their references and citations for more details. These differential conditions, usually called \emph{Killing spinor equations} in the literature \cite{Gran:2018ijr,Ortin,Tomasiello}, involve remarkably rich mathematical constructions that in many cases go beyond the type of spinors considered in the mathematical literature and that require specific new methods for their investigation. The method that we develop here, and that originated in the early supergravity literature on supersymmetric solutions \cite{Gauntlett:2002nw,Gibbons:1982fy,Tod:1983pm,Tod:1995jf}, is the theory of spinorial polyforms, which is based on the study of spinorial differential equations via an associated differential-algebraic system for a spinorial polyform, namely for the algebraic square of a spinor \cite{Cortes:2019xmk}. Despite the fact that supersymmetric solutions are of the utmost importance in the theoretical physics community, which has devoted thousands of papers to their study, their mathematical theory, as well as the mathematical theory of supergravity, is yet to be established as a mathematical discipline, see \cite{Cortes:2018lan,Figueroa-OFarrill:2015rfh,Figueroa-OFarrill:2017tcy,deMedeiros:2018ooy,Lazaroiu:2016iav,Lazaroiu:2016nbq,Lazaroiu:2016spz,Lazaroiu:2017qyr,LazaroiuShahbaziAGT,Lazaroiu:2021vmb} for recent progress in this direction. In a Lorentzian set-up, most of the problems studied traditionally in the mathematics literature \cite{MullerSanchez}, such as geodesic completeness, geodesic connectedness, global hyperbolicity, well-posedness, or initial data, just to name a few, remain completely open for supergravity solutions. With this motivation in mind, and as explained in more detail below, we apply the general theory of spinorial polyforms to the systematic study of three-dimensional supersymmetric solutions in NS-NS supergravity \cite{Ortin}, both because of their intrinsic interest as Lorentzian three-manifolds satisfying special curvature conditions, as well as a laboratory to understand the more complicated geometric and topological structure of these solutions in higher dimensions. Supersymmetric solutions of NS-NS supergravity involve irreducible spinors parallel with respect to a metric connection with totally skew torsion given by the curvature of a curving on a bundle gerbe. Although the study of Riemannian metric connections with totally skew torsion is classical in the mathematics literature \cite{Bismut,CristinaCastro,CleytonMoroianuSemmelmann}, also in connection with the study of supergravity Killing spinor equations in Riemannian signature \cite{Friedrich:2001nh,Garcia-Fernandez,Picard:2024tqg}, we are not aware of any systematic study in Lorentzian signature, with some exceptions \cite{Galaev,Murcia:2019cck}. We hope this note can contribute to their systematic mathematical study. 


\subsection*{Main results and outline}


We summarize in the following bullet points the main results of the article and the contents of each of its sections. 

\begin{itemize}
	\item In Section \ref{sec:SpinorsAsPolyforms} we study the algebraic square of a real irreducible Clifford module in signature $(p-q) \equiv_8 1$ and in Theorem \ref{thm:reconstruction} we characterize it as a solution of a semi-algebraic system of equations in a truncated model of the Kähler-Atiyah algebra underlying the given quadratic vector space. This characterization lies at the heart of the spinorial polyform formalism, since it is the first step towards achieving a complete description of spinors in terms of polyforms. The main novelty of this section with respect to \cite{Cortes:2019xmk}, where similar results are obtained in signature $(p-q) \equiv_8 0,2$, is that in the present case an irreducible Clifford module does not define an isomorphism of unital and associative real algebras. This leads to the use of the aforementioned \emph{truncated model} as already proposed in \cite{LazaroiuB,LazaroiuB}.
	
	\item In Section \ref{sec:spinorialdifferential} we introduce the notion of differential spinor and we develop the theory of spinorial polyforms associated to differential spinors in signature $(p-q) \equiv_8 1$, obtaining a complete equivalence in Theorem \ref{thm:GCKS}, which can now be applied to the study of every differential spinor in these signatures. In particular, we apply it to the three-dimensional Lorentzian case, proving in Theorem \ref{thm:differentialspinors3d} that every differential spinor in this dimension and signature is equivalent to a trivialized isotropic line preserved by a metric connection with torsion given explicitly in terms of $\cA$ and relative to a direction also explicitly given in terms of the latter.
	
	\item In Section \ref{sec:skewtorsion} we apply the previous formalism to the study of a very specific type of differential spinor in three Lorentzian dimensions, namely spinors parallel with respect to a connection with totally skew-symmetric torsion, to which we refer as \emph{skew-torsion parallel spinors}. We obtain a geometric characterization of this type of spinors via a class of adapted parallelisms that we call \emph{null coframes}, showing that they are necessarily Kundt \cite{Boucetta:2022vny,Kundt}, and we apply it to completely solve the problem locally. Equivalently, a Lorentzian three-manifold admits a skew-torsion spinor if and only if it admits an isotropic one-form $u\in \Omega^{1}(M)$ parallel under a metric connection with totally skew-torsion, possibly non-parallel. Although the case of parallel torsion has been recently studied in \cite{ErnstI,ErnstII}, the non-parallel case seems to no have been considered before in the literature. Several natural open problems arise in this context, for instance the geodesic completeness and classification problems of compact Lorentzian three-manifolds equipped with a skew-torsion parallel spinor, or equivalently, an isotropic vector parallel with respect to a connection with skew-symmetric torsion. The methods and results of \cite{HanounahMehidi} seem to be particularly promising to study this case. In this direction, we prove we obtain a simple condition that guarantees that the Levi-Civita connection of $(M,g)$ defines an affine structure on the leaves of the foliation canonically determined by the kernel of $u\in \Omega^1(M)$, encountering thus one of the cases considered in \cite{HanounahMehidi}. Furthermore, in the compact we case we obtain in Proposition \ref{prop:completeness} an explicit differential condition which, if it does not admit a solution, guarantees the geodesic completeness of the underlying Lorentzian three-manifold. 
	
	\item In Section \ref{sec:susygerbes} we consider the particular type of skew-torsion spinors that occurs in the study of three-dimensional NS-NS supersymmetric solutions. These are Lorentzian three-manifolds that satisfy a natural curvature condition, namely the Einstein equation of NS-NS supergravity, and are in addition equipped with a skew-torsion spinor relative to a metric connection with torsion given by the curvature of a bundle gerbe. This section exemplifies the main motivation behind our interest on differential spinors and the theory of spinorial polyforms: the study of supersymmetric solutions in supergravity, in this case NS-NS supergravity. Whereas NS-NS supergravity is formulated on a ten-dimensional manifold, we consider it here on a three-dimensional manifold as a first step towards the mathematical understanding of the differential geometry, topology and moduli of supersymmetric solutions in this ten-dimensional supergravity theory. Our main result in this section is Theorem \ref{thm:susysolutions}, which characterizes globally every three-dimensional supersymmetric solution in terms of an explicit differential system that involves a cohomological condition and for which the dilaton decouples. We apply this theorem to easily obtain the most general NS-NS supersymmetric solution in local adapted coordinates. The study of the Cauchy problem of these supersymmetric solutions, which is work in progress and will be consider elsewhere, provides a natural framework to generalize the results of \cite{Ammann,Glockle:2023ozz,Murcia:2020zig,Murcia:2021dur,RKS} on initial data sets for spinors on globally hyperbolic Lorentzian manifolds. 
\end{itemize}

\begin{ack} 
CSS would like to thank Vicente Cortés, Malek Hanounah, Lilia Mehidi, Alessandro Tomasiello and Abdelghani Zeghib for very useful discussions and comments, and Calin Lazaroiu for many discussions and debates on spin geometry over the years. The work of CSS was partially supported by the research Excellency María Zambrano grant and the Leonardo grant LEO22-2-2155 of the BBVA Foundation, as well as the Quantum Universe research grant of the University of Hamburg.
\end{ack}


\section{Algebraic spinorial polyforms in signature $(p-q)\equiv_8 1$}
\label{sec:SpinorsAsPolyforms}


Let $V$ be an oriented $d$-dimensional real vector space equipped with a non-degenerate metric $h$ of signature $(p-q)\equiv_8 1$, whence $V$ is of odd dimension $d=p+q$, and let $(V^{\ast},h^{\ast})$ be the quadratic space dual to $(V,h)$. Let $\Cl(V^{\ast},h^{\ast})$ be the universal real Clifford algebra associated to $(V^{\ast},h^{\ast})$, which in our conventions is defined via the relation $v^2 = h^{\ast}(v,v)$, $v\in V^{\ast}$. In the following we will denote by $\pi$ the standard automorphism of $\Cl(V^{\ast},h^{\ast})$, which acts as minus the identity on $V^{\ast}\subset \Cl(V^{\ast},h^{\ast})$, and we denote by $\tau$ the standard anti-automorphism of $\Cl(V^{\ast},h^{\ast})$, which acts as the identity on $V^{\ast}\subset \Cl(V^{\ast},h^{\ast})$. In signature $(p-q)\equiv_8 1$ the pseudo-Riemannian volume form $\nu_h$ on $(V,h)$ squares to the identity and belongs to the center of $\Cl(V^{\ast},h^{\ast})$. Therefore, $\Cl(V^{\ast},h^{\ast})$ is non-simple and splits as a direct sum of unital and associative algebras:
\begin{equation}
\label{eq:splittingCl}
\Cl(V^{\ast},h^{\ast}) = \Cl_{+}(V^{\ast},h^{\ast})  \oplus \Cl_{-}(V^{\ast},h^{\ast})\, , 
\end{equation}

\noindent
where:
\begin{equation*}
\Cl_{l}(V^{\ast},h^{\ast}) = \left\{ x\in \Cl(V^{\ast},h^{\ast}) \,\, \vert\,\, \nu_{h} x = l x\right\}\, , \quad l\in \mathbb{Z}_2\, .
\end{equation*}

\noindent
Equivalently:
\begin{equation*}
\Cl_{l}(V^{\ast},h^{\ast}) = \frac{1}{2} (1 + l \nu_h) (\Cl(V^{\ast},h^{\ast})) \, , \quad l\in \mathbb{Z}_2\, .
\end{equation*}

\noindent
The algebra $\Cl_{l}(V^{\ast},h^{\ast})$, $l\in \mathbb{Z}_2$, is simple and isomorphic to the algebra of $2^{\frac{d-1}{2}}$ by $2^{\frac{d-1}{2}}$ square matrices with real entries. Therefore, $\Cl(V^{\ast},h^{\ast})$ admits two irreducible left Clifford modules of real dimension $2^{\frac{d-1}{2}}$:
\begin{equation*}
\gamma_l\colon \Cl(V^{\ast},h^{\ast})\to \End(\Sigma)\, , \qquad l\in\mathbb{Z}_2\, ,
\end{equation*}

\noindent
that correspond to the projection of $\Cl(V^{\ast},h^{\ast})$ to the factor $\Cl_l(V^{\ast},h^{\ast})$ composed with an isomorphism of the latter to the algebra of endomorphisms $\End(\Sigma)$ of a real vector space $\Sigma$ of dimension $2^{\frac{d-1}{2}}$. These two Clifford modules are distinguished by the value they take at the volume form $\nu_h\in \Cl(V^{\ast},h^{\ast})$, namely:
\begin{equation*}
\gamma_l(\nu_h) = l\,\Id \in \End(\Sigma)\, .
\end{equation*}

\noindent
In particular, the kernel of $\gamma_l$ is given by:
\begin{equation*}
\mathrm{Ker}(\gamma_{+}) =  \Cl_{-}(V^{\ast},h^{\ast})\, , \qquad \mathrm{Ker}(\gamma_{-}) = \Cl_{+}(V^{\ast},h^{\ast}) 
\end{equation*}

\noindent
We will equip $\Sigma$ with an \emph{admissible pairing} \cite{AC,ACDP}, that is, an either symmetric or skew-symmetric non-degenerate real bilinear pairing $\cB$ on $\Sigma$ naturally compatible with $\gamma_l$, in sense that $\cB$ satisfies either:
\begin{equation}
\label{eq:admissiblepairings}
\cB(\gamma_l(x)(\xi_1),\xi_2) = \cB(\xi_1, \gamma_l(\tau(x))(\xi_2)) \quad \mathrm{or}\quad
\cB(\gamma_l(x)(\xi_1),\xi_2) = \cB(\xi_1, \gamma_l(\pi(\tau(x))(\xi_2))
\end{equation}
 
\noindent
for all $x\in \Cl(V^{\ast},h^{\ast})$ and $\xi_1, \xi_2 \in \Sigma$. In the first case we say that $\cB$ is of \emph{positive adjoint type} $\sigma = +1$ whereas in the second case we say that $\cB$ is of \emph{negative adjoint type} $\sigma = -1$. The symmetry type of $\cB$ as a bilinear form will be denoted by $s\in\mathbb{Z}_2$.

\begin{prop} 
\label{prop:admissiblepairings}
Every irreducible real Clifford module $\Sigma$ in signature $(p-q)\equiv_8 1$ admits an admissible pairing $\cB$ with symmetry and adjoint type given as follows in terms of the modulo $4$ reduction of $k := \frac{d-1}{2}$:
\begin{center}
\begin{tabular}{ | l | p{3cm} | p{3cm} | p{3cm} | p{3cm} |}
\hline
$k\,mod$ 4 & 0 & 1 & 2 & 3 \\ \hline
Symmetry type & Symmetric & Skew-symmetric & Skew-symmetric & Symmetric   \\ \hline
Adjoint type & Positive & Negative & Positive & Negative  \\ \hline
\hline
\end{tabular}
\end{center}

\

\noindent
In addition, if $\cB$ is symmetric, then it is of split signature unless $pq=0$, in which case $\cB$ is definite. 
\end{prop}

\begin{proof}
Pick an $h^\ast$-orthonormal basis $\left\{ e^{i}\right\}_{i=1,\ldots, d}$ of $V^\ast$ and define:
\begin{equation*}
\mathrm{K}(\left\{ e^{i}\right\}) := \{1\}\cup \left\{ \pm e^{i_1}\cdot \ldots \cdot e^{i_k}\, | \,\, 1\leq i_1<\ldots <i_k\leq d \, , \, 1\leq k\leq d\right\}\, ,
\end{equation*}

\noindent
to be the finite multiplicative subgroup of $\Cl(V^{\ast},h^{\ast})$ generated by the elements $\pm e^i$. Averaging over
$\mathrm{K}(\left\{ e^{i}\right\})$, we construct an auxiliary symmetric and positive-definite inner product $(-,-)$ on $\Sigma$ which is invariant under the action of $\mathrm{K}(\left\{ e^{i}\right\})$. By construction this product satisfies:
\begin{equation*}
(\gamma_l(x)(\xi_1), \gamma_l(x)(\xi_2)) = (\xi_1, \xi_2)\, , \qquad \forall \,\, x\in \mathrm{K}(\left\{e^{i}\right\})\,\, \forall \,\, \xi_1\, , \xi_2 \in \Sigma\, .
\end{equation*}

\noindent
Write $V^{\ast} = V^{\ast}_{+} \oplus V^{\ast}_{-}$, where $V^{\ast}_{+}$ is a $p$-dimensional subspace of $V^{\ast}$ on which $h^{\ast}$ is positive definite and $V^{\ast}_{-}$ is a $q$-dimensional subspace of $V^{\ast}$ on which $h^{\ast}$ is negative-definite. Fix an orientation on $V^{\ast}_{+}$, which induces a unique orientation on $V^{\ast}_{-}$ compatible with the orientation of $V^{\ast}$ induced from that of $V$, and denote by $\nu_{+}$ and $\nu_{-}$ the corresponding pseudo-Riemannian volume forms. We have $\nu = \nu_{+} \wedge \nu_{-}$.  We define:
\begin{equation*}
\label{eq:Bpmodd}
\cB(\xi_1 , \xi_2) = (\gamma_l(\nu_{+})(\xi_1),\xi_2) \qquad \forall \xi_1, \xi_2 \in \Sigma\, ,
\end{equation*}

\noindent
A direct computation gives:
\begin{equation*}
\cB(\gamma_l(v)\xi_1 , \xi_2) = - (-1)^{\frac{1+d}{2}} \cB(\xi_1 , \gamma_l(v)\xi_2)\, , \quad \cB(\xi_1 , \xi_2) =  (-1)^{\frac{d^2 -1}{8}} \cB(\xi_2 ,\xi_1)
\end{equation*}

\noindent
for every $v\in V^{\ast}$ and $\xi_1, \xi_2 \in \Sigma$. Substituting $d=2k+1$ in the previous equations we obtain the table in the statement. Using $\nu_{-}$ instead of $\nu_{+}$ to define $\cB$ yields the same bilinear pairing modulo a global sign. The fact that $\cB$ is definite if $pq =0$ holds by construction, whereas the fact that $\cB$ is of split signature if $pq \neq 0$ follows from the fact that we have either $\gamma_l(\nu_{+})^2 = 1$ or $\gamma_l(\nu_{+})^2 = -1$.
\end{proof}

\noindent
We will assume that $\Sigma$ is equipped with an admissible pairing of adjoint type $\sigma \in \mathbb{Z}_2$ and symmetry type $s\in \mathbb{Z}_2$. By construction we have:
\begin{equation*}
 \cB(\gamma(x)(\xi_1),\xi_2) + \cB(\xi_1,\gamma(x)(\xi_2)) = 0 \qquad \xi_1 , \xi_2 \in \Sigma
\end{equation*}

\noindent
for all $x = v_1\cdot v_2$ with $v_1, v_2 \in V^{\ast}$ such that $\vert h^{\ast}(v_1,v_1)\vert = \vert h^{\ast}(v_2,v_2)\vert = 1$ and $h^{\ast}(v_1,v_2) = 0$. This implies that $\cB$ is invariant under the action of identity component $\Spin_0(V^{\ast},h^{\ast})\subset \Spin(V^{\ast},h^{\ast})$ of the spin group $\Spin(V^{\ast},h^{\ast})$. For further reference we introduce the following definition. 

\begin{definition}
An \emph{irreducible paired Clifford module} is a triple $(\Sigma,\gamma_l,\cB)$ consisting of an irreducible Clifford module $(\Sigma,\gamma_l)$ equipped with an admissible bilinear pairing $\cB$ of symmetry type $s\in\mathbb{Z}_2$ and adjoint type $\sigma \in \mathbb{Z}_2$.
\end{definition} 
 
\noindent
The class of spinor bundles that we will consider in this article will be modeled on irreducible paired Clifford modules.


\subsection{The K\"ahler-Atiyah model of $\Cl(V^\ast,h^\ast)$}


In order to construct the \emph{square} of a spinor as a polyform, we will use the natural identification:
\begin{equation*}
\Psi\colon \Cl(V^{\ast},h^{\ast}) \to (\wedge V^{\ast} , \diamond_h) 
\end{equation*}

\noindent
between the $\Cl(V^{\ast},h^{\ast})$ and the K\"ahler-Atiyah algebra of $(V^{\ast},h^{\ast})$, which we denote by $(\wedge V^{\ast} , \diamond_g)$ \cite{Chevalley1,Chevalley2}. Here $\diamond_h\colon \wedge V^{\ast}\times \wedge V^{\ast} \to \wedge V^{\ast}$ denotes the \emph{geometric product} determined by $h$, which is explicitly given by the linear and associative extension of the following expression:
\begin{equation*}
v \diamond_h \alpha = v\wedge \alpha + \iota_{v^{\sharp_h}} \alpha \qquad \forall\,\, v \in
V^{\ast} \quad \forall\,\, \alpha\in \wedge V^{\ast}\, .
\end{equation*}

\noindent 
We transport the splitting \eqref{eq:splittingCl} of $\Cl(V^{\ast},h^{\ast})$ to $(\wedge V^{\ast},\diamond)$ through the natural identification $\Psi$, obtaining:
\begin{equation*}
(\wedge V^{\ast},\diamond_h) = (\wedge_{+} V^{\ast},\diamond_h) \oplus (\wedge_{-} V^{\ast},\diamond_h) 
\end{equation*}

\noindent
where $\wedge_{l} V^{\ast} = \left\{ \alpha\in \wedge V^{\ast} \,\, \vert\,\,  \nu_h \diamond_h \alpha = l\, \alpha\right\}$. We use the symbol $e_l$ to denote the unit in $\wedge_{l} V^{\ast}$, which is explicitly given by:
\begin{equation*}
e_l = \frac{1}{2} (1 + l \nu_h)\, .
\end{equation*}

\noindent 
Using the identity:
\begin{equation}
\label{eq:nuaction}
\alpha \diamond_h \nu_h  =   \nu_h \diamond_h \alpha   = \ast_h\, \tau(\alpha)\, , \qquad \forall \,\, \alpha \in \wedge V^{\ast}
\end{equation}
	
\noindent
which is useful in computations, we can equivalently write:
\begin{equation*}
\wedge_{l} V^{\ast} = \left\{ \alpha\in \wedge V^{\ast} \,\, \vert\,\,  \ast_h \tau(\alpha) = l\, \alpha\right\}\, .
\end{equation*}

\noindent
By applying the natural identification between $\Cl(V^{\ast},h^{\ast})$ and $\wedge V^{\ast}$ given by $\Psi$, the irreducible representation $\gamma_l$ defines a surjective morphism of unital associative real algebras between $ (\wedge V^{\ast}, \diamond)$ and $\End(\Sigma)$ that we denote by:
\begin{equation*}
\Psi_{\gamma_l} := \gamma_l\circ \Psi^{-1} \colon (\wedge V^{\ast}, \diamond_h)\to \End(\Sigma) \, .
\end{equation*}

\noindent
Let $\cP_l \colon \wedge V^{\ast} \to \wedge_l V^{\ast}$ be the natural projection of $\wedge V^{\ast}$ onto $\wedge_l V^{\ast}$. We have a canonical linear inclusion:
\begin{equation*}
\iota_l \colon \wedge_l V^{\ast} \hookrightarrow \wedge V^{\ast}
\end{equation*}

\noindent
which is a right inverse to $\cP_l$. In particular:
\begin{equation*}
	\Psi_{\gamma_l}\circ \iota_l \colon (\wedge_l V^{\ast}, \diamond_h)\to \End(\Sigma) \, .
\end{equation*}

\noindent
is an isomorphism of unital and associative algebras. Following \cite{LazaroiuB,LazaroiuBII,LazaroiuBC}, we define:
\begin{equation*}
\wedge^{<} V^{\ast} = \bigoplus_{k=0}^{\frac{d-1}{2}} \wedge^k V^{\ast}\subset \wedge V^{\ast}
\end{equation*}

\noindent
Note that $\wedge V^{\ast} = \wedge^{<} V^{\ast} \oplus \ast_h\wedge^{<} V^{\ast}$. Then, restricting the projection $\cP_l \colon \wedge V^{\ast} \to \wedge_l V^{\ast}$ to $\wedge^{<} V^{\ast}$:
\begin{equation*}
\cP_l\vert_{\wedge^{<} V^{\ast}} \colon \wedge^{<} V^{\ast} \to \wedge_l V^{\ast} 
\end{equation*}

\noindent
we obtain an isomorphism of vector spaces that we can use to transport the algebra product in $(\wedge_l V^{\ast},\diamond_h)$ to $\wedge^{<} V^{\ast}$. For every pair $\alpha_1, \alpha_2 \in \wedge^{<} V^{\ast}$ we define:
\begin{equation*}
\alpha_1 \vee_h \alpha_2 = (\cP_l\vert_{\wedge^{<} V^{\ast}})^{-1}(\cP_l\vert_{\wedge^{<} V^{\ast}}(\alpha_1)\diamond_h \cP_l\vert_{\wedge^{<} V^{\ast}}(\alpha_1)) =  (\cP_l\vert_{\wedge^{<} V^{\ast}})^{-1}(\cP_l(\alpha_1 \diamond_h \alpha_1))   = 2 \cP_{<}(\cP_l(\alpha_1 \diamond_h \alpha_1))
\end{equation*}

\noindent
where $\cP_{<}\colon \wedge V^{\ast} \to \wedge^{<} V^{\ast}$ is the natural projection of $\wedge V^{\ast}$ onto $\wedge^{<}V^{\ast}$. By construction, $(\wedge_l V^{\ast},\diamond_h)$ and $(\wedge^{<} V^{\ast},\vee_h)$ are naturally isomorphic as unital and associative real algebras. For further reference we introduce the following linear map:
\begin{equation*}
	\Psi^{<}_{\gamma_l} :=  \Psi_{\gamma_l} \circ \iota_l\circ  \cP_l \vert_{\wedge_l V^{\ast}} \colon (\wedge^{<} V^{\ast}, \vee_h)\to \End(\Sigma)  
\end{equation*}

\noindent
which, by the previous discussion, is an isomorphism of unital and associative algebras. All together, we obtain the following commutative diagram of unital and associative algebras:
\begin{eqnarray*}
	\scalebox{1.0}{
		\xymatrix{
			& \Cl_l(V,h) ~\ar[d]  \\
			\Cl(V,h) ~\ar[ur]^{p_l}\ar[r]^{\gamma_l} \ar[d]_{\Psi}& ~ \End(\Sigma) & \ar[l]_{\Psi^{<}_{\gamma_l}} \ar@<-2pt>[ld]_{\cP_l} (\wedge^{<} V^{\ast},\vee_h)  \\
			(\wedge V^{\ast},\diamond_h)~\ar[r]^{\cP_l} & \ar@<2pt>[l]^{\iota_l} ~  (\wedge_l V^{\ast},\diamond_h) \ar[u] \ar[ur]_{2\cP_{<}}}}
\end{eqnarray*}
 
\

\noindent
In particular, for every $\alpha \in \wedge V^{\ast}$ we have:
\begin{equation}
\label{eq:Cliffordrelation}
\gamma_l \circ \Psi^{-1}(\alpha) = 2\Psi^{<}_{\gamma_l} \circ \cP_{<}\circ \cP_l (\alpha) = \Psi^{<}_{\gamma_l} (\, \alpha^{<} + l \ast_g\tau(\alpha^{>}))
\end{equation}

\noindent
where $\alpha^{<} := \cP_{<}(\alpha)$ and $\alpha^{>} := \alpha - \alpha^{<}$. Note that $2\cP_{<}\colon \wedge_l V^{\ast} \to \wedge^{<} V^{\ast}$ is the inverse of the linear isomorphism $\cP_l\vert_{\wedge^{<} V^{\ast}}\colon \wedge^{<} V^{\ast} \to \wedge_l V^{\ast} $.

\begin{remark}
\label{remark:basis}
Let $(e^1,\hdots ,e^d)$ be an orhonormal basis of $(V^{\ast},h^{\ast})$ and let $\mathrm{I}\in \End(\Sigma)$ be the identity endomorphism of $\Sigma$. The following sets:
\begin{equation*}
\left\{1,\sqcup_{k=1}^{\frac{d-1}{2}}  e^{i_1}\wedge \cdots \wedge e^{i_k} \right\}\, , \quad \left\{e_l,\sqcup_{k=1}^{\frac{d-1}{2}}  \cP_l(e^{i_1}\wedge \cdots \wedge e^{i_k})\right\}\, , \quad \left\{\mathrm{I},\sqcup_{k=1}^{\frac{d-1}{2}}  \Psi^{<}_{\gamma_l}(e^{i_1}) \cdots \Psi^{<}_{\gamma_l} (e^{i_k})\right\}
\end{equation*}

\noindent
where it is assumed that $i_{1}<\cdots < i_k$, are basis of $\wedge^{<} V^{\ast}$, $\wedge_l V^{\ast}$ and $\End(\Sigma)$, respectively. 
\end{remark}

\noindent
Through the isomorphism $\Psi^{<}_{\gamma_l}$ the trace on $\End(\Sigma)$ transfers to the \emph{truncated} K\"ahler-Atiyah algebra $(\wedge^{<} V^{\ast}, \vee_g)$, defining the {\em truncated K\"ahler-Atiyah trace} $\cS_l$, explicitly given by the following linear map:
\begin{equation*}
\cS_l\colon  (\wedge^{<} V^{\ast}, \vee_h) \to \R\, , \qquad  \alpha\mapsto \Tr(\Psi^{<}_{\gamma_l}(\alpha))\, .
\end{equation*}

\noindent
Since $\Psi^{<}_{\gamma_l}$ is a unital morphism of algebras, we have:
\begin{equation*}
\cS_l(1) = 2^{\frac{d-1}{2}}\, ,   \qquad \cS_l(\alpha_1\vee_h \alpha_2) =  \cS_l(\alpha_2\vee_h\alpha_1) \qquad \forall\,\, \alpha_1,\alpha_2\in \wedge^{<} V^{\ast}
\end{equation*}

\noindent
where $1\in \wedge^0 V^{\ast}$ is the identity of $(\wedge^{<} V^{\ast},\vee_h)$.

\begin{prop}
\label{prop:tracev}
We have $\cS_l(\alpha) = 2^{\frac{d-1}{2}} \alpha^{(0)}$ for every $\alpha \in \wedge^{<} V^{\ast}$.
\end{prop}

\begin{proof}
Let $\left\{ e^i\right\}_{i=1,\ldots,d}$ be an orthonormal basis of $(V^{\ast},h^{\ast})$. For $i\neq j$ we have $e^i\vee_h e^j = - e^j\vee_h e^i$ and hence $(e^i)^{-1}\vee_h e^j\vee_h e^i = - e^j$. Let $1 \leq k \leq (d-1)/2$ and $1\leq i_1 < \cdots < i_k \leq d$. If $k$ is even, then:
\begin{equation*}
\cS_l(e^{i_1}\vee_h \cdots \vee_h e^{i_{k}}) = \cS_l(e^{i_k}\vee_h e^{i_1}\vee_h \cdots \vee_h e^{i_{k-1}}) = - \cS_l(e^{i_1}\vee_h \cdots \vee_h e^{i_{k}})
\end{equation*}

\noindent
and hence $\cS(e^{i_1}\vee_h \cdots \vee_h e^{i_{k}}) = 0$. Here we have used cyclicity of the K\"ahler-Atiyah trace and the fact that $e^{i_{k}}$ anticommutes with $e^{i_1}\vee_h \cdots\vee_h e^{i_{k-1}}$. If $k$ is odd, let $j\in \left\{1,\dots,d\right\}$ be such that $j\not \in \left\{i_1,\dots,i_k\right\}$, which exists since $k<d$. We have:
\begin{eqnarray*}
& \cS_l(e^{i_1}\vee_h \cdots \vee_h e^{i_{k}}) = - \cS((e^j)^{-1} \vee_h e^{i_1}\vee_h \cdots  \vee_h e^{i_{k}}\vee_h e^j) \\
& = - \cS_l(e^{i_1}\vee_h \cdots \vee_h e^{i_{k}}) = 0 
\end{eqnarray*}

\noindent
and hence we conclude.
\end{proof}

\noindent
The \emph{transpose} operation with respect to $\cB$ can also be also be nicely transported to $(\wedge^{<}V^{\ast},\vee_h)$ thanks to the admissability of $\cB$.
\begin{lemma}
\label{lemma:transposePsi}
The following formula holds for every $\alpha\in \wedge^{<} V^{\ast}$:
\begin{equation*}
\Psi_{\gamma_l}^{<}(\alpha)^t = \Psi_{\gamma_l}^{<}((\pi^{\frac{1-\sigma}{2}}\circ \tau)(\alpha))
\end{equation*}

\noindent
where $(-)^t$ denotes the adjoint operation with respect to $\cB$.
\end{lemma}
 
\begin{proof}
Follows directly from Equation \eqref{eq:admissiblepairings}.
\end{proof}

\noindent
Therefore, by the previous lemma we define the \emph{transpose} $\alpha^t\in \wedge^{<} V^{\ast}$ of an element $\alpha\in \wedge^{<} V^{\ast}$ by:
\begin{equation*}
\alpha^t = (\pi^{\frac{1-\sigma}{2}}\circ \tau)(\alpha)
\end{equation*}

\noindent
Note that with this definition we have $\Psi_{\gamma_l}^{<}(\alpha^t) = \Psi_{\gamma_l}^{<}(\alpha)^t$.

\begin{example}
If $d=1$ then $(\wedge^{<} V^{\ast}, \vee_h)$ is isomorphic to the algebra of real numbers $\mathbb{R}$. If $d=3$ and $(V,h)$ is of mostly plus signature, then:
\begin{equation*}
\wedge^{<} V^{\ast} = \mathbb{R} \oplus V^{\ast}
\end{equation*}

\noindent
and:
\begin{equation*}
(c_1 + \alpha_1) \vee_h (c_2 + \alpha_2)  = 2 (\cP_{<}\circ\cP_l)((c_1 + \alpha_1) \diamond_h (c_2 + \alpha_2)) = c_1 c_2 + c_1 \alpha_2 + c_2  \alpha_1 + h^{\ast}(\alpha_1,\alpha_2) -\ast_h (\alpha_1\wedge\alpha_2)
\end{equation*}

\noindent
where $c_1 + \alpha_1 , c_2 + \alpha_2 \in \mathbb{R} \oplus V^{\ast}$. This algebra is isomorphic to the algebra of square two by two matrices with real entries.
\end{example}


\subsection{Spinor squaring maps}


Given an irreducible paired Clifford module $(\Sigma, \gamma_l, \cB)$ and a sign $\mu \in \mathbb{Z}_2$, we introduce the following quadratic map:
\begin{equation*}
\cE^{\mu} \colon \Sigma \to \End(\Sigma)\, , \qquad \varepsilon\mapsto \mu \,\varepsilon\otimes \varepsilon^{\ast}\, ,
\end{equation*}

\noindent
where $\varepsilon^{\ast} := \cB( - ,\varepsilon)\in \Sigma^{\ast}$ is the dual of $\varepsilon$ defined through $\cB$.

\begin{definition}
\label{def:squarespinor}
Let $(\Sigma, \gamma_l, \cB)$ be an irreducible paired Clifford module. The {\em spinor square map} associated to $(\gamma_l, \Sigma, \cB)$ is the quadratic map:
\begin{equation*}
\label{eq:spinorsquaremap}
\cE^{\mu}_{\gamma_l} := (\Psi_{\gamma_l}^{<})^{-1} \circ \cE^{\mu} \colon \Sigma \to \wedge^{<} V^{\ast}\, ,
\end{equation*}

\noindent
Given a spinor $\varepsilon\in \Sigma$, the polyforms $\cE^{+}_{\gamma_l}(\varepsilon)$ and $\cE^{-}_{\gamma_l}(\varepsilon)$ are, respectively, the positive and negative {\em squares} of $\varepsilon$ relative to the admissible pairing $\cB$.  
\end{definition}
 
\noindent
We will say that a polyform $\alpha\in \wedge^{<} V^{\ast}$ is the {\em square} of $\varepsilon\in \Sigma$ if either $\alpha=\cE^{+}_{\gamma_l}(\varepsilon)$ or $\alpha=\cE^{-}_{\gamma_l}(\varepsilon)$. Note that the spinor square map is equivariant with respect to the natural action of $\Spin_0(V^{\ast},h^{\ast})$ on the source $\Sigma$ and target $\wedge^{<} V^{\ast}$, respectively. The following result gives the algebraic characterization of spinors in terms of polyforms that we will use to study differential spinors in signature $(p-q)\equiv_8 1$.

\begin{thm}
\label{thm:reconstruction} 
Let $(\Sigma,\gamma_l,\cB)$ be an irreducible paired Clifford module of symmetry type $s$ and adjoint type $\sigma$. The following statements are equivalent for a polyform $\alpha\in \wedge^{<} V^{\ast}$:
\begin{enumerate} 
\itemsep 0.0em
\item $\alpha$ is a the square of a spinor $\varepsilon \in \Sigma$, namely $\alpha \in \mathrm{Im}(\cE^{+}_{\gamma_l})\cup \mathrm{Im}(\cE^{-}_{\gamma_l})$.  

\item $\alpha$ satisfies the following relations:
\begin{equation}
\label{eq:thmdefequationsequiv}
\alpha\vee_h\alpha =\cS(\alpha) \, \alpha\, , \quad  (\pi^{\frac{1-\sigma}{2}}\circ\tau)(\alpha) = s\,\alpha\, , \quad 
\alpha \vee_h \beta \vee_h \alpha = \cS(\alpha \vee_h \beta)\, \alpha\, ,
\end{equation}

\noindent
for a fixed polyform $\beta \in \wedge^{<} V^{\ast}$ such that $\cS(\alpha\vee_h\beta) \neq 0$.

\item The following relations hold:
\begin{equation}
\label{eq:thmdefequations}
(\pi^{\frac{1-\sigma}{2}}\circ\tau)(\alpha) = s\,\alpha\, , \quad  \alpha\vee_h \beta \vee_h \alpha = \cS(\alpha \vee_h \beta)\, \alpha
\end{equation}
for every polyform $\beta\in \wedge^{<} V^{\ast}$.
\end{enumerate}

\noindent
In particular, the image of $\cE^{\mu}_{\gamma_l}$ in $\wedge_l V^{\ast}$ depends only on $s$, $\sigma$ and $h^{\ast}$. 
\end{thm}

\begin{proof}
Let $\xi\in\Sigma - \left\{ 0 \right\}$ and set $E = \cE^{\mu}(\xi)\in \End(\Sigma)$. Then $E$ satisfies:
\begin{equation}
\label{eq:rankone}
E\circ E = \mathrm{Tr}(E) E\, , \qquad E^t = s E\, .
\end{equation}

\noindent
and more generally:
\begin{equation*}
E\circ T\circ  E = \mathrm{Tr}(E\circ T) E
\end{equation*}

\noindent
for every $T\in \End(\Sigma)$. Applying $\Psi_{\gamma_l}^{<}$ to the previous equations and using that it is a morphism of unital and associative algebras, together with Lemma \ref{lemma:transposePsi}, we obtain that $(1) \Rightarrow (2)$ and $(1) \Rightarrow (3)$.

Conversely, by \cite[Proposition 2.21]{Cortes:2019xmk}, if $E\in\End(\Sigma)$ is any endomorphism satisfying the previous system of equations, then there exists an element $\xi\in \Sigma$, unique modulo a sign, such that $E = \cE^{\mu}(\xi)$. Applying now $(\Psi^{<}_{\gamma_l})^{-1} := \colon \End(\Sigma) \to (\wedge_l V^{\ast}, \diamond_h)$ to the system of equations \eqref{eq:rankone} and using that $\Psi^{<}_{\gamma_l}$ is an isomorphism of unital and associative algebras we conclude.
\end{proof}

\noindent
Note that by Remark \ref{remark:basis} every endomorphism $E\in \End(\Sigma)$ can be expanded as follows:
\begin{equation*}
E = \frac{E_0}{2^{\frac{d-1}{2}}} \, \mathrm{I} + \frac{1}{2^{\frac{d-1}{2}}} \sum_{k=1}^{\frac{d-1}{2}} \sum_{i_1 < \dots < i_k} E_{i_1\hdots i_k} \Psi_{\gamma_l}^{<}(e^{i_1})\cdots \Psi_{\gamma_l}^{<}(e^{i_k})
\end{equation*}

\noindent
where $E_0 , E_{i_1\cdots i_k}\in \mathbb{R}$ and:
\begin{equation*}
E_{i_1\cdots i_k} = \mathrm{Tr}(\Psi_{\gamma_l}^{<}(e^{i_k})^{-1} \cdots \Psi_{\gamma_l}^{<}(e^{i_1})^{-1} E)
\end{equation*}

\noindent
Hence, if we set $E = \cE^{\mu}(\xi)$ for a $\xi\in\Sigma$ unique modulo a sign, we obtain:
\begin{equation*}
\mathrm{Tr}(\Psi_{\gamma_l}^{<}(e^{i_k})^{-1} \cdots \Psi_{\gamma_l}^{<}(e^{i_1})^{-1} E) = \mu\, \cB(\Psi_{\gamma_l}^{<}(e^{i_k})^{-1} \cdots \Psi_{\gamma_l}^{<}(e^{i_1})^{-1}\xi,\xi)
\end{equation*}

\noindent
and from this formula it follows that every polyform in the image of the spinor square map $\cE^{\mu}_{\gamma_l}$ admits the following explicit presentation \cite{LazaroiuB, LazaroiuBII, LazaroiuBC}:
\begin{equation}
\label{eq:bilinears}
\alpha = \frac{\mu}{2^{\frac{d-1}{2}}} \sum_{k=0}^{\frac{d-1}{2}}  \,\sum_{i_1 < \dots < i_k} \cB(\Psi_{\gamma_l}^{<}(e^{i_k})^{-1} \cdots \Psi_{\gamma_l}^{<}(e^{i_1})^{-1}\xi,\xi)\, e^{i_1}\wedge \hdots \wedge e^{i_k}\, ,
\end{equation}

\noindent
where $\xi\in \Sigma$ is determined by $\alpha$ up to sign.
 

\subsection{Signature $(2,1)$}
\label{sec:Lorentzianexample3d}


Let $(V,h)$ denote the oriented three-dimensional Minkowski vector space with metric $h$ of signature $(2,1)$. Denote by $(V^{\ast},h^{\ast})$ its dual. The associated irreducible Clifford module $(\gamma_l , \Sigma)$ is two-dimensional and is equipped with an admissible pairing $\cB$ which by Proposition \ref{prop:admissiblepairings} is skew-symmetric and of negative adjoint type. To compute the square of an element $\varepsilon\in \Sigma$ we consider Equation \eqref{eq:thmdefequationsequiv} in Theorem \ref{thm:reconstruction}. A polyform $\alpha\in \wedge^{<} V^{\ast}$ satisfies the second equation in \eqref{eq:thmdefequationsequiv} if and only if $\alpha = u $ for a uniquely determined one-form  $u \in V^{\ast}$. A direct computation shows that $u$ satisfies the first equation in \eqref{eq:thmdefequationsequiv} if and only if $\vert u\vert^2_h = 0$. Consider now a isotropic vector $v\in V^{\ast}$ satisfying $h(u,v) = 1$. Then:
\begin{equation*}
\cS_l(v\vee_h u) = 2 \neq 0
\end{equation*}

\noindent
and hence such $v\in V^{\ast}$ satisfies the criteria to be chosen as $\beta$ for the third equation of \eqref{eq:thmdefequationsequiv}. The latter becomes:
\begin{equation*}
u\vee_h v\vee_h u = 2 u
\end{equation*}
 
\noindent
which is automatically satisfied. Hence, a polyform $\alpha \in \wedge^{<} V^{\ast}$ is the square of a non-zero irreducible spinor in signature $(2,1)$ if and only if $\alpha = u$ for a non-zero null vector. 


\section{Differential spinors in signature $(p-q)\equiv_8 1$}
\label{sec:spinorialdifferential}


To define and study differential spinors on pseudo-Riemannian manifolds of signature $(p-q)\equiv_8 1$, we \emph{globalize} the algebraic theory of spinors and polyforms laid out in Section \ref{sec:SpinorsAsPolyforms} to bundles of real irreducible Clifford modules equipped with an arbitrary connection.


\subsection{General theory}


Let $(M,g)$ denote a connected and oriented pseudo-Riemannian manifold of signature $(p,q)$ and odd dimension $d=p+q$, where $p-q\equiv_8 1$. Since $M$ is connected, the pseudo-Euclidean vector bundle $(TM,g)$ is modeled on a fixed quadratic vector space that we denote by $(V,h)$. For any point $m\in M$, we thus have an isomorphism of quadratic spaces $(T_mM,g_m)\simeq (V,h)$. Accordingly, the cotangent bundle $T^\ast M$ endowed with the dual metric $g^\ast$ is modeled on the dual quadratic space $(V^{\ast},h^{\ast})$. We denote by $\Cl(M,g)$ the bundle of real Clifford algebras of the {\em cotangent} bundle $(T^\ast M,g^\ast)$, which is modeled on the real Clifford algebra $\Cl(V^{\ast},h^{\ast})$. We use the notation introduced in Section \ref{sec:SpinorsAsPolyforms} to denote the corresponding objects naturally extended on $M$. 
 
\begin{definition}
A {\em bundle of real Clifford modules} on $(M,g)$ is a pair $(S,\Gamma)$, where $S$ is a real vector bundle on $M$ and
$\Gamma:\Cl(M,g)\rightarrow \End(S)$ is a unital morphism of bundles of unital and associative real algebras.
\end{definition}

\noindent 
Given a bundle of Clifford modules $(S,\Gamma)$, for every point $m\in M$ the unital morphism of associative algebras $\Gamma_m:\Cl(T_m^\ast M, g_m^\ast) \to \End(S_m)$ defines a Clifford module that we denote by $(S_m,\Gamma_m)$. Since $M$ is connected, it follows that there exists a Clifford module $(\Sigma,\gamma)$, unique modulo isomorphism, such that that $(S_m,\Gamma_m)$ is isomorphic to $(\Sigma,\gamma)$ via an \emph{unbased} isomorphism of Clifford modules \cite{Lazaroiu:2016vov}. In this situation we say that $(S,\Gamma)$ is a bundle of real Clifford modules of \emph{type} $(\Sigma,\gamma)$.

\begin{definition}
An \emph{paired spinor bundle} of type $(\Sigma,\gamma)$ is a triple $(S,\Gamma,\cB)$ consisting of a bundle of real Clifford modules $(S,\Gamma)$ of type $(\Sigma,\gamma)$ equipped with an admissible bilinear pairing $\cB$ on $S$ of symmetry type $s$ and adjoint type $\sigma$. An \emph{irreducible paired spinor bundle} is a paired spinor bundle of irreducible type $(\Sigma,\gamma)$. In the latter case, a global section $\epsilon\in \Gamma(S)$ is an \emph{irreducible spinor} on $(M,g)$.
\end{definition}

\noindent 
In the signature $p-q\equiv_8 1$ considered in this note, an irreducible spinor bundle has rank $2^{\frac{d-1}{2}}$, where $d$ is the dimension of $M$. References \cite{Lazaroiu:2016vov,LS2018} prove that $(M,g)$ admits an irreducible real spinor bundle if and only if it admits a {\em real Lipschitz structure} of irreducible type. As mentioned in the introduction, in signature $p - q\equiv_8 1$, a real Lipschitz structure of irreducible type corresponds to a standard spin structure $Q_g$ such that ($S,\Gamma)$ is the standard spinor bundle associated to $Q_g$ via the tautological representation of the spin group \cite{Lazaroiu:2016vov}.  

\begin{remark}
Admissible bilinear pairings on an irreducible spinor bundle may be obstructed. They are however guaranteed to exist if $(M,g)$ is \emph{strongly spin}, that is, if every spin structure on $(M,g)$ admits a reduction to the identity component of the spin group in the given signature.
\end{remark}

\begin{definition}
\label{def:generalizedKS}
Let $(S,\Gamma,\cB)$ be a paired spinor bundle on $(M,g)$ and let $\cD$ be a connection on $S$. A section $\epsilon \in \Gamma(S)$ is a \emph{differential spinor} on $(M,g)$ relative to $\cD$ if:
\begin{equation*}
\cD\varepsilon = 0
\end{equation*}

\noindent
Such $\varepsilon$ is a \emph{constrained differential spinor} relative to $(\cD,\cQ)$ if in addition:
\begin{equation*}
\cQ(\varepsilon) = 0
\end{equation*}

\noindent
where $\cQ\in \Gamma(\End(S))$ is an endomorphism of $S$.
\end{definition}

\begin{remark}
Supersymmetric solutions of supergravity theories are generally characterized as manifolds admitting certain systems of constrained differential spinors, see for instance \cite{LazaroiuB,LazaroiuBII}. This extends the notion of generalized
Killing spinor considered \cite{BarGM,FriedrichKim,FriedrichKimII,MoroianuSemm}.
\end{remark}

\noindent
Let $(S,\Gamma,\cB)$ is an irreducible paired spinor bundle. Since  $(S,\Gamma,\cB)$ is associated to a spin structure, we can write $\cD=\nabla^g-\cA$ for a unique element $\cA\in \Omega^1(\End(S))$, where $\nabla^g$ is the lift of the Levi-Civita connection to $S$. In this case, the equations satisfied by a constrained differential spinor can be equivalently written as:
\begin{equation*}
\nabla^g\epsilon = \cA (\epsilon)\, , \qquad  \cQ(\epsilon) = 0
\end{equation*}

\noindent
Their solutions are called constrained differential spinors {\em relative to $(\cA,\cQ)$}. Note that, using connectedness of $M$ and the parallel transport defined by $\cD$, it follows that the set of constrained differential spinors relative to $(\cA,\cQ)$ is a finite-dimensional vector space. As explained in Section \ref{sec:SpinorsAsPolyforms}, associated to $(S,\Gamma)$ we obtain an isomorphism of bundles of unital and associative algebras:
\begin{equation*}
\Psi_{\Gamma}^{<}  \colon (\wedge^{<} T^{\ast}M, \vee_g)\to \End(S)  
\end{equation*}

\noindent
as well as a \emph{quadratic map}: 
\begin{equation*}
\label{eq:spinorsquarebundlemap}
\cE^{\mu}_{\Gamma}  \colon S \to \wedge^{<} T^{\ast}M 
\end{equation*}

\noindent
which we extend pointwise to sections $\Gamma(S)$ of $S$ and denote by the same symbol for ease of notation. Evaluating $\cA$ on a vector $v\in\mathfrak{X}(M)$ we obtain a field of endomorphisms $\cA_v \in \Gamma(\End(S))$ to which we can apply the inverse of $\Psi_{\Gamma}^{<}$ in order to obtain a section of $\wedge^{<} T^{\ast}M$. We define:
\begin{equation*}
\fra_v := (\Psi_{\Gamma}^{<})^{-1}(\cA_v)\, , \quad \forall \,\, v\in\mathfrak{X}(M)
\end{equation*}

\noindent
and we denote by $\fra\in\Omega^1(\End(S))$ the associated one-form taking values in $\wedge^{<} T^{\ast}M$. Following \cite{Meinrenken}, we refer to $\fra$ as the \emph{dequantization} of $\cA$. Similarly, given any field of endomorphisms $\cQ\in \Gamma(\End(S))$, we define $\frq := (\Psi_{\Gamma}^{<})^{-1}(\cQ)$.

\begin{thm}
\label{thm:GCKS}
A strongly spin Lorentzian manifold $(M,g)$ of signature $(p-q)\equiv_8 1$ admits a constrained differential spinor of type $(\Sigma,\gamma_l)$ relative to a pair $(\cA,\cQ)$ with given dequantization $(\fra,\frq)$ if and only if there exists a nowhere vanishing polyform $\alpha\in \Gamma(\wedge^{<} M)$ such that the following differential system is satisfied: 
\begin{equation}
\label{eq:Differentialpolyformsystem}
\nabla^g\alpha = \fra\vee_g \alpha + \alpha\vee_g (\pi^{\frac{1-\sigma}{2}}\circ\tau)(\fra)\, , \qquad  \frq \vee_h \alpha  = 0 
\end{equation}

\noindent
and either the following algebraic equations are satisfied for every polyform $\beta\in \Gamma(\wedge_l M)$:
\begin{equation}
\label{eq:FierzglobalI}
\alpha \vee_h \beta \vee_h \alpha = \cS(\alpha \vee_h \beta)\,  \alpha\, , \qquad (\pi^{\frac{1-\sigma}{2}}\circ\tau)(\alpha) = s\,\alpha 
\end{equation}

\noindent
or, equivalently, the following equations:
\begin{equation}
\label{eq:FierzglobalII}
\alpha \vee_h \alpha = \cS(\alpha) \, \alpha\, , \qquad (\pi^{\frac{1-\sigma}{2}}\circ\tau)(\alpha) =  s\,\alpha\, , \qquad \alpha \vee_h\beta \vee_h \alpha = \cS(\alpha \vee_h \beta)\, \alpha
\end{equation}

\noindent
are satisfied for some fixed polyform $\beta\in \Gamma(\wedge^{<} M)$ such that $\cS(\alpha\diamond\beta) \neq 0$.
\end{thm}

\begin{proof}
Once Theorem \ref{thm:reconstruction} as been established, the proof is completely analogous to that of \cite[Theorem 4.26]{Cortes:2019xmk}. Since the proof of the latter is relatively technical and long, we omit it for simplicity in the exposition. The only key point to notice is that to prove the theorem in the present case we need identify the endomorphism bunlde $\End(S)$ of $S$ with the \emph{truncated model} of the K\"ahler-Atiyah bundle that we introduced in detail in the previous sections, and that we have denoted by $(T^{\ast}M,\vee_g)$, instead of identifying $\End(S)$ with the full K\"ahler-Atiyah bundle as in the proof of \cite[Theorem 4.26]{Cortes:2019xmk}.  
\end{proof}

\begin{remark}
If $\cA$ is skew-symmetric with respect to $\cB$, then \eqref{eq:Differentialpolyformsystem} simplifies to:
\begin{equation*}
\label{eq:GKSII}
\nabla^g\alpha = \fra\vee_h \alpha -  \alpha\vee_h \fra\, .
\end{equation*}

\noindent
Note that, in applications to supergravity, $\cA$ need {\em not} be skew-symmetric relative to any admissible pairing $\cB$. 
\end{remark}

\begin{remark}
Every polyform $\alpha\in \Omega^{\bullet}(M)$ acts naturally via Clifford multiplication, which in our framework is given in terms of $\Psi$ and $\Gamma$. For ease of notation we denote it by:
\begin{eqnarray*}
\alpha\cdot \varepsilon := \Gamma\circ \Psi^{-1}(\alpha)(\varepsilon) = \Psi_{\Gamma}^{<}(\,\alpha^{<} + l \ast_g\tau(\alpha^{>}))(\varepsilon)
\end{eqnarray*}

\noindent
where we have used Equation \eqref{eq:Cliffordrelation}. Note that $\Psi^{-1}(\alpha)\in \Gamma(\Cl(M,g))$ is naturally a section of the bundle of Clifford algebras $\Cl(M,g)$.
\end{remark}

\noindent
In the following we will apply the previous theorem to study differential spinors on Lorentzian three-manifolds.


\subsection{Differential spinors on Lorentzian three-manifolds}


Let $(M,g)$ be an oriented and time-oriented Lorentzian three-manifold. Denote by $\nabla^g$ the Levi-Civita connection on $(M,g)$. Every other metric connection $\nabla$ on $(M,g)$ can be written as follows:
\begin{equation*}
	\nabla_{w_1}w_2 = \nabla^g_{w_1} w_2 + A^g(w_1,w_2)\, , \qquad w_1 , w_2 \in \mathfrak{X}(M)
\end{equation*}

\noindent
in terms of a uniquely defined tensor $A^g\in \Gamma(T^{\ast}M\otimes T^{\ast}M\otimes TM)$ satisfying:
\begin{equation*}
	g(A^g(w_1,w_2),w_3) +g(w_2, A^g(w_1,w_3)) = 0
\end{equation*}

\noindent
for every $w_1,w_2,w_3\in\mathfrak{X}(M)$. For ease of notation we use the same symbol to define $A^g(w_1,w_2,w_3) := g(A^g(w_1,w_2),w_3)$, in terms of which the previous condition becomes:
\begin{equation*}
	A^g(w_1,w_2,w_3) + A^g(w_1,w_3,w_2) = 0
\end{equation*}

\noindent
Hence $A^g\in \Gamma(T^{\ast}M\otimes \wedge^2 T^{\ast}M)$. We will refer to such tensor $A^g$ as the \emph{metric contorsion tensor} of the connection $\nabla$ on $(M,g)$. The vector space of all metric contorsion tensors on $(M,g)$ identifies with all sections $\Gamma(T^{\ast}M\otimes \wedge^2 M)$ of the bundle $T^{\ast}M\otimes \wedge^2 M$. 

Every paired spinor bundle $(S,\Gamma,\cB)$ on $(M,g)$ is of rank two and equipped with a skew-symmetric non-degenerate bilinear pairing of negative adjoint type. Given $(S,\Gamma,\cB)$ and $\cA\in\Omega^1(M,\End(S))$, we denote by $\fra\in \Omega^1(M,\wedge^{<} M)$ the dequantization of the latter, which we expand as follows:
\begin{equation}
\label{eq:decompositioncA}
\fra  =  \frac{1}{2} (\theta +  \chi )
\end{equation}

\noindent
where $\theta \in \Omega^1(M)$ and $\chi\in \Gamma(T^{\ast}M\otimes T^{\ast}M)$. When necessary, we will denote by $\chi_w$ the evaluation of a vector field $w\in\mathfrak{X}(M)$ on the first entry of $\chi$. The previous expression captures the dequantization of the possible $\cA\in \Omega^1(M,\End(S))$ relative to which differential spinors can exist on $(M,g)$. 

\begin{thm}
\label{thm:differentialspinors3d}
A spin Lorentzian three-manifold $(M,g)$ admits a differential spinor relative to a paired spinor bundle $(S,\Gamma,\cB)$ and an endomorphism-valued one-form $\cA\in \Omega^1(M,\End(S))$ with dequantization $\fra\in\Omega^1(M,\wedge^{<} M)$ if and only if it admits a isotropic vector field $u\in \Omega^1(M)$ satisfying:
\begin{equation*}
 \cD^{g,\chi} u = \theta\otimes u
\end{equation*}

\noindent
where  $\cD^{g,\chi}$ is the unique metric connection on $(M,g)$ with contorsion tensor given by:
\begin{equation*}
A^g_{w_1}(w_2) = l \ast_g(\chi_{w_1}\wedge w_2)
\end{equation*}

\noindent
where  $w_1, w_2 \in \mathfrak{X}(M)$.
\end{thm}
 
\begin{remark}
The one-form $u\in\Omega^1(M)$, or its metric dual $u^{\sharp_g}\in\mathfrak{X}(M)$, is usually referred to as the \emph{Dirac current} of $\varepsilon\in \Gamma(S)$. 
\end{remark}

\begin{proof}
 By Theorem \ref{thm:GCKS} $(M,g)$ admits a differential spinor if and only if Equations \eqref{eq:Differentialpolyformsystem} and \eqref{eq:FierzglobalI} (equivalently, Equation \eqref{eq:FierzglobalII}) are satisfied. Equation \eqref{eq:FierzglobalI} is satisfied if and only if $\alpha = u\in \Omega^1(M)$ for an isotropic one-form on $(M,g)$. Equation \eqref{eq:Differentialpolyformsystem} then becomes:
 \begin{equation*}
 \nabla^g_w u = \fra_w \vee_g u + u\vee_g (\pi \circ\tau)(\fra_w) = \frac{1}{2} (\theta_w + \chi_w) \vee_g u + \frac{1}{2} u \vee_g (\theta_w - \chi_w) = \theta_w u + \frac{1}{2}(\chi_w  \vee_g u - u \vee_g \chi_w)
 \end{equation*}
 
 \noindent
 We further compute:
 \begin{equation*}
 \chi_w  \vee_g u = l\ast_g (u\wedge \chi_w) + g(\chi_w,u)
 \end{equation*}
 
 \noindent
 and thus:
 \begin{equation*}
 \frac{1}{2}(\chi_w  \vee_g u - u \vee_g \chi_w)= l\ast_g (u\wedge \chi_w) 
 \end{equation*}
 
 \noindent
 from which the result follows.
\end{proof}

\noindent
In other words, a differential spinor on $(M,g)$ defines a bundle of isotropic lines trivialized by $u$ and preserved by $\cD^{g,\chi}$ in the \emph{direction} prescribed by $\theta$. Hence, every differential spinor defines an \emph{optical structure} with torsion that can be studied within the general framework developed in \cite{Fino:2020uuz}.  By Lemma \ref{lemma:transposePsi} it immediately follows that $\cA$ is skew, that is, the connection $(\cD - \cA)$ on $S$ preserves $\cB$, if and only if $\theta = 0$ in the decomposition \eqref{eq:decompositioncA}, namely, if and only if the Dirac current $u\in \Omega^1(M)$ is actually parallel with respect to $\cD^{g,\chi}$. In this case, the optical structure determined by $\varepsilon$ reduces to a Bargmannian structure with torsion as described in \cite{Figueroa-OFarrill:2020gpr}.

\begin{remark}
\label{remark:algebraicconstraint}
If a differential spinor is constrained to be in the kernel an endomorphism $\cQ\in\Gamma(\End(S))$ with fixed dequantization $\frq \in \Omega^1(\wedge^{<} M)$ then we simply need to add to the statement of the previous theorem the condition $\frq \vee_g u = 0$, where $u$ is the Dirac current associated to the given differential spinor.
\end{remark}
 
\noindent
Alternatively, we can rephrase the previous theorem in the following form, which is sometimes more convenient for computations.

\begin{cor}
A strongly spin Lorentzian Lorentzian three-manifold $(M,g)$ admits a differential spinor if and only if there exists a isotropic one-form $u\in \Omega^1(M)$ satisfying:
\begin{equation*}
\nabla^g_w u = \theta_w u + \ast_g (u\wedge \chi_w)\, , \qquad \forall\,\, w\in \mathfrak{X}(M)
\end{equation*}

\noindent
for a one-form $\theta\in \Omega^1(M)$ and a tensor $\chi\in \Gamma(T^{\ast}M\otimes T^{\ast}M)$.
\end{cor}

\noindent
Note that, without loss of generality, in the previous corollary we have absorbed the sign $l$ into $\chi$. Theorem \ref{thm:differentialspinors3d} characterizes all possible differential spinors on an oriented and spin Lorentzian three-manifold $(M,g)$ through an adequate choices of $\theta$ and $\chi$, and hence recovers well-known types of special spinors on $(M,g)$ as particular cases. More precisely:

\begin{itemize}
	\item If $\theta = 0$ and $\chi = 0$ both vanish identically then $\varepsilon$ is a parallel spinor and we conclude that $(M,g)$ admits a parallel irreducible real spinor field if and only if it admits a isotropic vector field parallel with respect to the Levi-Civita connection. Note that this does not longer hold in higher dimensions \cite{Murcia:2020zig,Murcia:2021dur}.
	
	\item If $\theta = 0$ and $\chi = \lambda g$ for a non-zero real constant $\lambda$ then $\varepsilon$ is a standard Killing spinor and we conclude that $(M,g)$ admits an irreducible real Killing spinor if and only if it admits a isotropic vector field $u\in \mathfrak{X}(M)$ satisfying:
	\begin{equation}
	\label{eq:Killing}
	\nabla^g u =  \lambda\,  \ast_g  u 
	\end{equation}

	\noindent
	The sign factor $l\in\mathbb{Z}_2$ has been absorbed in $\lambda$ without further consequence.
	
	\item If $\theta = 0$ and $\chi\in \End(TM)$ is an endomorphism of $TM$ symmetric with respect to $g$ then $\varepsilon$ is formally an irreducible Lorentzian generalized Killing spinor, namely a type of Lorentzian analog of the notion of generalized Killing spinor introduced in \cite{MoroianuSemm,MoroianuSemmI,MoroianuSemmII} for Riemannian manifolds. In particular, $(M,g)$ admits a Lorentzian generalized Killing spinor if and only if it admits a isotropic vector field satisfying:
	\begin{equation*}
	\nabla^g_w u =   \ast_g (u\wedge \chi_w)\, , \qquad \forall\,\, w\in \mathfrak{X}(M)
	\end{equation*}
	
	\noindent
	However, this is not, strictly speaking, a generalized Killing spinor as it was originally introduced \cite{BarGM,FriedrichKim0,FriedrichKimII}, namely it is not the restriction of an irreducible parallel spinor to a hypersurface. This is because a real irreducible spinor in four Lorentzian dimensions does not descend to an irreducible spinor in three Lorentzian dimensions. Hence, a generalized Killing spinor in three Lorentzian dimensions would correspond a \emph{faithful} real spinor satisfying differential equation of the type above. To the best of our knowledge, this type of generalized Killing spinors, which cannot occur in Riemannian signature, have not been studied in the literature. 
	
	\item If $\theta = 0$ and $\chi \in \Omega^2(M)$ is a two-form, then $\varepsilon$ is a \emph{skew-torsion parallel spinor}, namely a spinor parallel under a connection with totally skew-symmetric torsion $H\in\Omega^3(M)$ given by:
	\begin{eqnarray*}
	H(w_1,w_2,w_3) = -2  (\ast_g\chi_{w_1})(w_2,w_3)\, , \qquad w_1, w_2, w_3 \in \mathfrak{X}(M)
	\end{eqnarray*}
	
	\noindent
	In particular, $(M,g)$ admits a skew-torsion parallel spinor with respect to the spinorial lift of a metric connection with totally skew-symmetric torsion $\nabla^{g,H}$ if and only if it admits a isotropic vector field parallel under the very same connection $\nabla^{g,H}$. We will further explore this case in Section \ref{sec:skewtorsion}, since it is the main object of study in this note.
	
	\item If $\chi = 0$ then the corresponding $\cA\in \Omega^1(M,\End(S))$ takes values in the line spanned by the identity isomorphism, that is, $\varepsilon\in \Gamma(S)$ satisfies:
	\begin{equation*}
	\nabla^g\varepsilon = \theta\otimes\varepsilon 
	\end{equation*}
	
	\noindent
	for a one-form $\theta\in\Omega^1(M)$. Hence, Theorem \ref{thm:differentialspinors3d} implies in this case that $(M,g)$ admits such a differential spinor if and only if:
	\begin{equation*}
	\nabla^g u = \theta\otimes u
	\end{equation*}

	\noindent
	that is, if and only if $(M,g)$ admits a recurrent isotropic line, namely a bundle of isotropic lines preserved by the Levi-Civita connection of $(M,g)$. Manifolds equipped with such a recurrent isotropic line, called \emph{Walker manifolds} in \cite{WalkerManifolds} and \emph{weakly abelian} in \cite{MehidiZeghib}, have been already considered in the literature, see for instance \cite{Chaichi,Galaev:2009ie,Galaev:2010jg} and their references and citations. Interestingly enough, general differential spinors are characterized similarly in Theorem \ref{thm:differentialspinors3d} in terms of a bundle of isotropic lines preserved by $\cD^{g,\chi}$ instead of the Levi-Civita connection.
\end{itemize}

\begin{remark}
For dimensional reasons a real Killing spinor in three Lorentzian dimensions is a particular case of a skew-torsion parallel spinor. Indeed, using the identity $\ast_g  u =  u\lrcorner_g \nu_g$ it follows that Equation \eqref{eq:Killing} is equivalent to:
\begin{equation*}
\nabla^g_w u + \lambda\,  u\lrcorner_g w\lrcorner_g \nu_g = 0
\end{equation*}

\noindent
Hence, the Dirac current $u$ associated to a Killing spinor is parallel with respect to the unique metric connection on $(M,g)$ with completely skew torsion given by $2\lambda \nu_g\in \Omega^3(M)$, where $\nu_g$ is the Lorentzian volume form on $(M,g)$. In contrast, a general skew-torsion parallel spinor is parallel with respect to the spinorial lift of a metric connection with torsion $\frf \nu_g\in \Omega^3(M)$ for a possibly not constant function $\frf\in C^{\infty}(M)$. As we will see in the following, the theory of skew-torsion parallel spinors with non-parallel torsion is markedly more complicated than the parallel torsion case.
\end{remark}

\noindent
Given $u\in\mathfrak{X}(M)$ we choose a conjugate isotropic vector field $v\in\mathfrak{X}(M)$, which by definition satisfies $g(u,v) =1$. We call such a pair $(u,v)$ a \emph{conjugate pair}. Given a conjugate pair $(u,v)$, we obtain a canonical, unit norm, space-like vector field:
\begin{eqnarray*}
n = - \ast_g (u\wedge v)
\end{eqnarray*}

\noindent
making the triple $(u,v,n)$ positively oriented. We will refer to such a triplet $(u,v,n)$ as a \emph{null coframe} of $M$. The line spanned by $n$ in $T^{\ast}M$ is the \emph{screen distribution} determined by $(u,v)$, which in this case reduces to a rank-one distribution and is therefore always integrable. For further reference we note the following identities, which are useful in computations:
\begin{equation}
\label{eq:dualhodges}
\ast_g u = - u\wedge n\, , \qquad \ast_g v = v\wedge n\, , \qquad \ast_g n = u\wedge v \, , \qquad \ast_g\nu _g = -1 \, .
\end{equation}

\noindent
Isotropic vector fields conjugate to a fixed isotropic vector field $u\in\Omega^1(M)$ are unique modulo transformations of the form:
\begin{equation*}
v \mapsto v - \frac{F^2}{2} u + f n
\end{equation*}

\noindent
for a function $F\in C^{\infty}(M)$. We denote by $\cP_u$ the set of null coframes associated to a given isotropic one-form $u\in\Omega^1(M)$. Then $\cP_u$ is a torsor over $C^{\infty}(M)$ with the following action:
\begin{equation*}
F\cdot (u,v,n) = (u, v -\frac{F^2}{2} u + F n , n - F u)
\end{equation*}

\noindent
where $F\in C^{\infty}(M)$ and $(u,v,n)\in \cP_u$. Note that for every pair of null coframes $(u,v,n) , (u,v^{\prime},n^{\prime}) \in \cP_u$ we have:
\begin{equation*}
g = u\odot v + n\otimes n = u\odot v^{\prime} + n^{\prime}\otimes n^{\prime}\, .
\end{equation*}

\noindent
We will make extensive use of null coframes in the following sections.


\section{Skew-Torsion parallel spinors}
\label{sec:skewtorsion}


In this section we focus on Lorentzian three-manifolds equipped with a particular class of differential spinors that we call \emph{skew-torsion parallel spinors}. These are irreducible real spinors parallel under the lift to the spinor bundle of a metric connection with totally skew-symmetric torsion. 


\subsection{General characterization}


In this subsection we use Theorem  \ref{thm:differentialspinors3d} to describe skew-torsion parallel spinors in terms of an equivalent differential system for a null coframe associated to the corresponding Dirac current.

\begin{prop}
\label{prop:torsionparallel}
A three-dimensional strongly spin Lorentzian manifold $(M,g)$ admits a skew-torsion parallel spinor if and only if it admits an isotropic one-form $u\in \Omega^1(M)$ satisfying:
\begin{equation}
\label{eq:torsionparallel}
\nabla^g u = \frf \ast_g u
\end{equation}

\noindent
for a function $\frf\in C^{\infty}(M)$.  
\end{prop}

\begin{remark}
\label{remark:Kundttorsion}
Equivalently, $u\in\Omega^1(M)$ is parallel under the unique metric connection on $(M,g)$ with totally skew-torsion $2 \,\frf\,\nu_g \in \Omega^3(M)$. In particular, it immediately follows that if $(M,g)$ is a Lorentzian three-manifold equipped with a skew-torsion parallel spinor with associated Dirac current $u\in \Omega^1(M)$, then $(M,g,u)$ is a Kundt three-dimensional space-time \cite{Boucetta:2022vny,Kundt}. Hence, this defines a particular class of Kundt three-manifolds that, to the best of our knowledge, has not been studied in the literature and certainly deserves more attention.
\end{remark}

\begin{proof}
The result follows directly from Theorem \ref{thm:differentialspinors3d}, but for completeness we prove it directly by means of Theorem \ref{def:generalizedKS}. Let $\varepsilon\in\Gamma(S)$ be a spinor parallel under the lift to the spinor bundle of a metric connection with totally skew-symmetric torsion. Then, $\varepsilon$ satisfies:
\begin{equation*}
\nabla^g_w\varepsilon + \frac{1}{4}\Gamma\circ \Psi^{-1}(H_w)(\varepsilon)= 0\, , \qquad w\in \mathfrak{X}(M)
\end{equation*}

\noindent
where $H\in \Omega^3(M)$ is a three-form. Hence, $\varepsilon$ is a differential spinor relative to an endomorphism-valued one-form $\cA\in\Omega^1(\End(S))$ given by:
\begin{eqnarray*}
\cA_w = - \frac{1}{4}\Gamma\circ \Psi^{-1}(H_w) \in \Gamma(\End(S))\, , \qquad w\in \mathfrak{X}(M)
\end{eqnarray*}

\noindent
Therefore:
\begin{equation*}
\fra_w = (\Psi_{\Gamma}^{<})^{-1}(\cA_w ) = - \frac{1}{8} (\cP_l\vert_{\wedge_l M})^{-1} (H_w + l \nu_g \diamond_g H_w)= - \frac{1}{4} \cP_{<} (H_w - l \ast_g H_w) = \frac{l}{4} \ast_g H_w
\end{equation*}

\noindent
which, by Theorem \ref{thm:differentialspinors3d}, implies:
\begin{equation*}
\nabla^g_w u =  \frac{l}{4} (\ast_g H_w\vee_g u - u \vee_g \ast_g H_w) = \frac{1}{2} H(w,u)
\end{equation*}

\noindent
Equivalently, $u$ is parallel with respect to the unique metric connection with totally skew-symmetric torsion given by $-H\in\Omega^3(M)$. Setting $H = -2\,\frf\, \nu_g$ in terms of a uniquely determined function $\frf\in C^{\infty}(M)$, we have:
\begin{equation*}
H(w,u) = -2\, \frf\, u\lrcorner_g w\lrcorner_g \nu_g = 2\, \frf\, w\lrcorner_g \ast_g u
\end{equation*}

\noindent
and hence we conclude.
\end{proof}

\noindent
By the previous result, it immediately follows that the Dirac current associated to a skew-torsion parallel spinor is Killing and its orthogonal distribution is integrable and defines a codimension one smooth foliation by null hypersurfaces. Furthermore, we have $\nabla^g_u u = 0$ whence the integral curves of $u$ define a non-expanding, non-twisting and non-shear null geodesic congruence. Therefore, a Lorentzian three-manifold $(M,g)$ equipped with a torsion parallel spinor defines an idealized gravitational wave that belongs to the Kundt class of space-times. Projecting Equation \eqref{eq:torsionparallel} to its symmetric and skew-symmetric components we immediately obtains the following corollary.

\begin{cor}
\label{cor:torsionparalleldiff}
A three-dimensional strongly spin Lorentzian manifold $(M,g)$ admits a skew-torsion parallel spinor if and only if it admits an isotropic Killing vector field $u^{\sharp_g}\in\mathfrak{X}(M)$ whose dual one-form $u\in \Omega^1(M)$ satisfies:
\begin{equation}
	\label{eq:torsionparalleldiff}
\frac{1}{2} \dd u= \frf \ast_g u
\end{equation}

\noindent
for a function $\frf\in C^{\infty}(M)$.  
\end{cor}

\noindent
For further reference we introduce the following definition. 

\begin{definition}
A skew-torsion Lorentzian three-manifold is a tuple $(M,g,u,\frf)$ consisting of a strongly spin Lorentzian three-manifold $(M,g)$ equipped with an isotropic one-form $u\in \Omega^1(M)$ satisfying Equation \eqref{eq:torsionparallel} with respect to the function $\frf\in C^{\infty}(M)$.
\end{definition}

\noindent
By Proposition \ref{prop:torsionparallel}, every skew-torsion Lorentzian carries a skew-torsion parallel spinor, which is canonical modulo a global sign.

\begin{lemma}
\label{lemma:conjugatecovarianttorsion}
The covariant derivative of a null coframe $(u,v,n)\in \cP_u$ on a skew-torsion Lorentzian three-manifold $(M,g,u,\frf)$ is given by:
\begin{eqnarray}
\label{eq:derivativeconjugateparallelismtorsion}
\nabla^g u =  \frf  \ast_g u = \frf\,n\wedge u\, , \qquad  \nabla^g v =     \kappa\otimes n -  \frf\, n\otimes v \, , \qquad \nabla^g n =  \frf \,  u\otimes v - \kappa\otimes u 
\end{eqnarray} 

\noindent
for a one-form $\kappa\in \Omega^1(M)$.
\end{lemma}

\begin{proof}
A direct calculation shows that $\ast_g u = n\wedge u$ and therefore the first equation in \eqref{eq:derivativeconjugateparallelismtorsion} follows from Proposition \ref{prop:torsionparallel}. For the remaining equations, we compute by taking the covariant derivative of \eqref{eq:dualhodges} with respect to any $w\in \mathfrak{X}(M)$:
\begin{eqnarray}
& \ast_g\nabla_{w}^g u = - \nabla_{w}^g u\wedge n - u\wedge \nabla_{w}^g n \, , \qquad \ast_g\nabla_{w}^g v = \nabla_{w}^g v \wedge n + v\wedge \nabla_{w}^g n \nonumber\\ 
& \ast_g \nabla_{w}^g n = \nabla_{w}^g u\wedge v + u\wedge \nabla_{w}^g v \label{eq:covariantidentities}
\end{eqnarray}

\noindent
Plugging $\nabla^g_w u =  \frf \, w\lrcorner\ast_g u = \frf \ast_g (u\wedge w) = \frf\,n(w) u - \frf\,u(w) n$ into the first equation above yields:
\begin{equation*}
\frf \, u\wedge w = u\wedge (\frf \, n(w) n + \nabla^g_w n)
\end{equation*}

\noindent
and thus:
\begin{equation*}
\nabla^g n = \frf\, u\otimes v + (\kappa_o + \frf\, v)\otimes u
\end{equation*}

\noindent
for a one-form $\kappa_o \in \Omega^1(M)$. Plugging this equation together with the covariant derivative of $u$ into the second line of \eqref{eq:covariantidentities}, we obtain:
\begin{equation*}
\nabla^g v =  - (\kappa_o + \frf\, v)\otimes n - \frf\, n\otimes v 
\end{equation*}

\noindent
Setting $\kappa:= - (\kappa_o + \frf\, v)$ we conclude. 
\end{proof}

\noindent
Let $(u,v,n)\in \cP_u$ be a null coframe satisfying the differential system \eqref{eq:derivativeconjugateparallelismtorsion} with respect to a one-form $\kappa\in \Omega^1(M)$. Then, any other null coframe $(u,v^{\prime},n^{\prime})\in \cP_u$ satisfies again the \eqref{eq:derivativeconjugateparallelismtorsion} with respect to the one-form:
\begin{equation}
\label{eq:kapparelation}
\kappa^{\prime} = \kappa + \dd F + \frf\, F n - \frac{\frf F^2}{2} u  
\end{equation}

\noindent
and hence, the choice of null coframe in $\cP_u$ is a matter of taste or convenience. 

\begin{prop}
\label{prop:exteriorconjugateparallelismtorsion}
A null coframe $(u,v,n)\in\cP_u$ satisfies equations \eqref{eq:derivativeconjugateparallelismtorsion} with respect to $\kappa\in \Omega^1(M)$ if and only if:
\begin{eqnarray}
	\label{eq:exteriorconjugateparallelismtorsion}
	\frac{1}{2} \dd u =\frf  \ast_g u = \frf\,n\wedge u\, , \qquad  \dd v =  ( \frf v + \kappa) \wedge n\, , \qquad \dd n = u\wedge (\frf v + \kappa) 
\end{eqnarray}

\noindent
for the same one-form $\kappa\in \Omega^1(M)$.
\end{prop}

\begin{proof}
The skew-symmetrization of equations \eqref{eq:derivativeconjugateparallelismtorsion} yields equations \eqref{eq:exteriorconjugateparallelismtorsion} for the same one-form $\kappa\in \Omega^1(M)$. On the other hand, the symmetrization of equations \eqref{eq:derivativeconjugateparallelismtorsion} is equivalent to:
\begin{eqnarray*}
\cL_{u^{\sharp_g}}g = 0\, , \qquad \cL_{v^{\sharp_g}}g =  \kappa\odot n -  \frf\, n\odot v \, , \qquad \cL_{n^{\sharp_g}}g = \frf \,  u\odot v - \kappa\odot u   
\end{eqnarray*}

\noindent
Assuming equation \eqref{eq:exteriorconjugateparallelismtorsion}, this equations are automatically satisfied and hence we conclude.
\end{proof}

\noindent
By the first equation in \eqref{eq:exteriorconjugateparallelismtorsion}, the distribution $u^{\perp_g}\subset TM$ determined by the kernel of $u \in \Omega^1(M)$ in $TM$ is integrable and defines a foliation $\cX_u$ of $M$ whose leaves are immersed two-dimensional submanifolds. We will refer to $\cX_u$ as the \emph{canonical foliation} of the given skew-torsion Lorentzian three-manifold $(M,g,u,\frf)$. By the first equation in \eqref{eq:exteriorconjugateparallelismtorsion}, the Godbillon-Vey class of $\cX_u$ is given explicitly by:
\begin{equation*}
\sigma_u = 4[\frf n\wedge \dd (\frf n)] = 4 [\frf^2 n\wedge u\wedge (\frf v+\kappa)] \in H^3(M,\mathbb{R})
\end{equation*}

\noindent
in terms of any null coframe $(u,v,n) \in \cP_u$. We will refer to $\sigma_u\in H^3(M,\mathbb{R})$ as the \emph{Godbillon-Vey class} of the given skew-torsion Lorentzian three-manifold $(M,g,u,\frf)$. Taking the exterior derivative of equations \eqref{eq:exteriorconjugateparallelismtorsion} we immediately the following exterior differential system:
\begin{equation*}
\label{eq:integrabilityconditions}
\dd\frf\wedge u\wedge n = 0\, , \qquad n \wedge ( \dd\frf\wedge v + \dd\kappa)   = 0\, , \qquad  (\dd \frf(n) + \frf^2 + \kappa(u) \frf)\, u \wedge v\wedge n = u\wedge \dd \kappa 
\end{equation*}

\noindent
to which we will refer as the \emph{integrability conditions} of the null coframe $(u,v,n)\in \cP_u$.  

\begin{prop}
\label{prop:dkappa}
Let $(u,v,n)\in \cP\in u$ be a null coframe satisfying equations \eqref{eq:derivativeconjugateparallelismtorsion} with respect to $\kappa\in\Omega^1(M)$. Then:
\begin{equation*}
\dd\kappa =\dd\frf(v) v\wedge u +  (\dd \frf(n) + \frf^2 + \kappa(u) \frf) v\wedge n + \mathfrak{l}\,  u\wedge n
\end{equation*}

\noindent
in terms of a function $\mathfrak{l}\in C^{\infty}(M)$.
\end{prop}

\begin{proof}
We expand:
\begin{eqnarray*}
\dd\kappa = \dd\kappa(u,v) v\wedge u + \dd\kappa(u,n) v\wedge n + \dd\kappa(v,n) u\wedge n
\end{eqnarray*}

\noindent
in terms of the given null coframe $(u,v,n)$. Substituting this expression into the integrability equations and isolating for the coefficients of the expansion we obtain the expression in the statement of the proposition. 
\end{proof}

\noindent
We end this subsection summarizing the equivalent descriptions of skew-torsion Lorentzian three-manifolds contained in propositions \ref{prop:torsionparallel} and \ref{prop:exteriorconjugateparallelismtorsion}.

\begin{cor}
\label{cor:conjugateparalleltorsion}
A three-dimensional strongly spin Lorentzian manifold $(M,g)$ admits a skew-torsion parallel spinor if and only if it admits a null coframe $(u,v,n)\in \cP_u$ satisfying equations \eqref{eq:derivativeconjugateparallelismtorsion}, or, equivalently, equations \eqref{eq:exteriorconjugateparallelismtorsion}, for some one-form $\kappa\in \Omega^1(M)$.
\end{cor}


\subsection{Curvature}

 
The curvature tensor of a skew-torsion Lorentzian three-manifold $(M,g)$ is highly constrained. We compute in the following both the Riemmann and Ricci curvatures of $g$, which in three-dimensions are equivalent curvature tensors. Denote by $\dd_{\nabla^g}\colon \Omega^1(M,\wedge^1 M) \to \Omega^2(M,\wedge^1 M)$ the exterior covariant derivative associated to the Levi-Civita connection $\nabla^g$. Using equations \eqref{eq:derivativeconjugateparallelismtorsion} together with the integrability conditions of the given null coframe $(u,v,n)$, we compute:
\begin{eqnarray*}
&\dd_{\nabla^g}\nabla^g u = \frf^2 u\wedge v \otimes u + u\wedge n \otimes (\dd \frf + \frf^2 n) = (\dd \frf\wedge n + \frf^2 u\wedge v)\otimes u - (\dd\frf\wedge u + \frf^2 n\wedge u)\otimes n\\
&\dd_{\nabla^g}\nabla^g v = (\dd\kappa + \frf\, n\wedge \kappa)\otimes n - (\dd \frf\wedge n + \frf^2 u\wedge v) \otimes v  \\
&\dd_{\nabla^g}\nabla^g n = (\dd \frf\wedge u + \frf^2 n\wedge u) \otimes v - (\dd\kappa + \frf\, n\wedge \kappa)\otimes u    
\end{eqnarray*}

\noindent
These expressions immediately imply that the Riemann tensor $\cR^g\in \Gamma(\wedge^2 M \otimes \wedge^2 M)$ of the Lorentzian metric $g=u\odot v + n\otimes n$ is given by:
\begin{eqnarray*}
\cR^g =  (\dd\frf \wedge n + \frf^2 u \wedge v)\otimes v\wedge u - (\dd\frf \wedge u + \frf^2 n\wedge u)\otimes v\wedge n + (\dd\kappa + \frf\, n\wedge \kappa )\otimes u\wedge n
\end{eqnarray*}

\noindent
Written in this form, the celebrated symmetries of the Riemann tensor are not apparent. Using the integrability conditions of the given null coframe, $\cR^g$ can be equivalently written as follows:
\begin{eqnarray*}
& \cR^g =  \dd\frf(v)  (u\wedge n\otimes v\wedge u + v\wedge u\otimes u\wedge n) + \frf^2 u \wedge v \otimes v\wedge u\\
& - (\dd\frf(n) +  \frf^2  )(n\wedge u \otimes v\wedge n  + v\wedge n \otimes  n\wedge u) + (\frf\, \kappa(v) - \mathfrak{l}) n\wedge u\otimes u\wedge n
\end{eqnarray*}

\noindent
from which the Ricci tensor of $g=u\odot v + n\otimes n$ can be computed to be:
\begin{eqnarray}
\label{eq:Riccicurvature}
\mathrm{Ric}^g = - \dd\frf(v) u\odot n - (2\frf^2 + \dd\frf(n)) g - \dd\frf(n)\, n\otimes n + (\frf\, \kappa(v) - \mathfrak{l})   u\otimes u 
\end{eqnarray}

\noindent
From this formula we immediately obtain that the scalar curvature of a skew-torsion Lorentzian three-manifold is given by:
\begin{eqnarray*}
\mathrm{s}^g =- (6\frf^2 + 4 \dd\frf(n))   
\end{eqnarray*}

\noindent
These explicit expressions for the Ricci and scalar curvatures of a skew-torsion Lorentzian three-manifold imply the following corollary.

\begin{cor}
\label{cor:Einsteincondition}
A skew-torsion Lorentzian three-manifold is Einstein if and only if and only if $\frf$ is constant and $\frf\, \kappa(v) = \mathfrak{l}$.
\end{cor}

\noindent
The fact that $(M,g)$ is Einstein only if the torsion is parallel, points out to the Einstein condition not being natural for general skew-torsion Lorentzian three-manifolds. In Section \ref{sec:susygerbes} we will encounter a curvature condition naturally compatible with skew-torsion Lorentzian three-manifolds in the context of supersymmetric solutions of supergravity.  

\begin{example}
\label{ep:exampleGodbillonVey}
The Godbillon-Vey class of a skew-torsion Lorentzian three-manifold is only relevant if $M$ is compact, since otherwise $H^3(M,\mathbb{R}) = \left\{ 0\right\}$. Setting $\kappa = \frf v$, the differential system \eqref{eq:exteriorconjugateparallelismtorsion} reduces to:
\begin{eqnarray*}
\frac{1}{2} \dd u = \frf\,n\wedge u\, , \qquad  \frac{1}{2}\dd v = \frf \, v \wedge n\, , \qquad \frac{1}{2}\dd n = \frf \, u\wedge v\, .
\end{eqnarray*}

\noindent
Assuming now that $\frf$ is in addition a non-zero constant, the first line in the previous differential system is satisfied by every right invariant, appropriately normalized, coframe on $\mathrm{Sl}(2,\mathbb{R})$. Hence, every right-invariant coframe on $\mathrm{Sl}(2,\mathbb{R})$ can be normalized as to define a null coframe $(u,v,n)$ associated to skew-torsion parallel spinor on $\mathrm{Sl}(2,\mathbb{R})$ with respect to the metric $g = u\odot v + n\otimes n$. In particular, $\mathrm{Sl}(2,\mathbb{R})$, as well as its many discrete quotients, admit plenty of skew-torsion parallel spinors with respect to left-invariant Lorentzian metrics. The corresponding Godbillon-Vey invariant is given by:
\begin{eqnarray*}
\sigma_u =  8 \frf^3 [ u\wedge v\wedge n] \in H^3(M,\mathbb{R})
\end{eqnarray*}

\noindent
and hence by appropriately choosing the constant $\frf$ we can realize every possible Godbillon-Vey invariant.
\end{example}

\begin{remark}
Going beyond the Godbillon-Vey class, there is a class of \emph{secondary invariants} associated to codimension-one foliations, such as the invariant introduced by Losik \cite{LosikI,LosikII,BazaikinGalaev}. It would be interesting to study these secondary invariants for the codimension-one foliations ocurring in skew-torsion Lorentzian three-manifolds.
\end{remark}


\subsection{Local adapted coordinates}


In the following we solve the differential system \eqref{eq:derivativeconjugateparallelismtorsion} in local coordinates adapted to the existence of a skew-torsion parallel spinor on $(M,g)$. 

\begin{lemma}
Let $(M,g,u,\frf)$ be a skew-torsion Lorentzian three-manifold. Then, there exist local coordinates $(x_u,x_v,z)$ in which the metric $g$ takes the form:
\begin{eqnarray*}
g = \cH\, \dd x_u\otimes \dd x_u + e^{\cF} \dd x_u\odot \dd x_v + \dd z \otimes \dd z
\end{eqnarray*}

\noindent
in terms of local functions $\cH$ and  $\cK$ depending only on the coordinates $x_u$ and $z$.  
\end{lemma}
 
\begin{proof}
Choose local coordinates $(x^{\prime}_v,y_1,y_2)$ such that $u^{\sharp_g} = \partial_{x^{\prime}_v}$. Since $\partial_{x^{\prime}_v}$ is isotropic and Killing, in these coordinates the metric $g$ reads:
\begin{equation*}
g = g_{i x^{\prime}_v} \dd x^{\prime}_v \odot \dd y_i + g_{ij}\, \dd y_i \otimes \dd y_j
\end{equation*}

\noindent
where the coefficients depend only on $(y_1,y_2)$. Since $u\wedge \dd u = 0$, the Frobenius theorem implies that $u = e^{\cF} \dd \kappa$ in terms of some locally defined functions $\cF$ and $\kappa$. In particular:
\begin{eqnarray*}
e^{\cF} \dd \kappa = g_{1 x^{\prime}_v} \dd y_1 + g_{2 x^{\prime}_v} \dd y_2
\end{eqnarray*}

\noindent
Since $u\neq 0$, we can assume, perhaps after relabelling the coordinates, that $g_{1 x^{\prime}_v} \neq 0$ locally around $m\in M$. Hence, $(x^{\prime}_v,\kappa,y_2)$ are local functions around $m\in M$ with linearly independent differentials and thus define a local coordinate system around $m\in M$. Setting $x_u:= \kappa_2$ the metric is locally given by:
\begin{eqnarray*}
g = \cH^{\prime}\, \dd x_u\otimes \dd x_u + e^{\cF} \dd x_u\odot \dd x^{\prime}_v + F\, \dd x_u\odot \dd y_2   +  e^{2\cK} \dd y_2 \otimes \dd y_2
\end{eqnarray*}

\noindent
for a local functions $\cH^{\prime}$, $\cF$, $\cK$ and $F$ depending only on  $x_u$ and $y_2$. Defining now the coordinate:
\begin{equation*}
z = \int e^{\cK} \dd y_2
\end{equation*}

\noindent
we obtain:
\begin{eqnarray*}
g = \cH^{\prime\prime}\, \dd x_u\otimes \dd x_u + e^{\cF} \dd x_u\odot \dd x^{\prime}_v + F^{\prime}\, \dd x_u\odot \dd z   +  \dd z \otimes \dd z
\end{eqnarray*}

\noindent
in terms of new functions $\cH^{\prime\prime}$ and $F^{\prime}$ depending only on $x_u$ and $z$. Introducing now the new coordinate:
\begin{eqnarray*}
x_v := x^{\prime}_v + \int e^{-\cF} F^{\prime} \dd z
\end{eqnarray*}

\noindent
we eliminate the \emph{crossed term} $F^{\prime}\, \dd x_u\odot \dd z$ and we conclude after a relabelling the coefficients of the metric. Note that $u^{\sharp_g} = \partial_{x^{\prime}_v} = \partial_{x_v}$.
\end{proof}

\noindent
Since the isotropic vector field $u^{\sharp_g}\in \mathfrak{X}(M)$ in local adapted coordinates is given by $u^{\sharp_g} = \partial_{x_v}$, we have $u = e^{\cF} \dd x_u$. A canonical choice of local vector field $v^{\sharp_g}$ that is locally conjugate to $u^{\sharp_g}$ is given by:
\begin{equation*}
v^{\sharp_g} = e^{-\cF} \partial_{x_u} - \frac{1}{2} \cH e^{-2\cF} \partial_{x_v}
\end{equation*}

\noindent
whence $v = \frac{1}{2} \cH_{x_u} e^{-\cF}  \dd x_u + \dd x_v$ in local adapted coordinates. From this, it immediately follows that:
\begin{equation}
\label{eq:localconjugate}
(u = e^{\cF} \dd x_u ,v = \frac{1}{2}\cH e^{-\cF}  \dd x_u + \dd x_v, n = e^{\cK} \dd z ) \in \cP_u
\end{equation}

\noindent
is a local null coframe associated to $u$. In particular, we have:
\begin{equation*}
n^{\sharp_g} = e^{-\cK} \partial_z
\end{equation*}

\noindent
for the metric dual of $n$.  
 
\begin{lemma}
Let $(M,g,u,\frf)$ be a skew-torsion Lorentzian three-manifold. Then, the local null coframe \eqref{eq:localconjugate} satisfies:
\begin{eqnarray}
\label{eq:generallocalequations}
\partial_z \cF = 2\frf \, , \qquad \partial_z \cH - \frf \,\cH +  2 \kappa_u e^{\cF} = 0\, , \qquad  \kappa_z = 0
\end{eqnarray}

\noindent
where we have written locally $\kappa = \kappa_{x_u} \dd x_u + \kappa_{x_v} \dd x_v + \kappa_{z} \dd z$. Furthermore,  $\kappa_v = - \frf$.
\end{lemma}

\begin{proof}
$(M,g,u,\frf)$ is a skew-torsion Lorentzian three-manifold if and only if the null coframe \eqref{eq:localconjugate} satisfies the differential system \eqref{eq:exteriorconjugateparallelismtorsion}, which, evaluated on \eqref{eq:localconjugate} is equivalent to  the equations in the statement of the lemma.
\end{proof}

\begin{prop}
\label{prop:localskewtorsion}
Let $(M,g,u,\frf)$ be a skew-torsion Lorentzian three-manifold. Around every point $m\in M$ there exist local coordinates $(x_u,x_v,z)$ in which the metric $g$ is given by:
\begin{eqnarray*}
g = \cH\, \dd x_u\otimes \dd x_u + e^{2\int \frf\,\dd z + c(x_u)} \dd x_u\odot \dd x_v + \dd z \otimes \dd z
\end{eqnarray*}

\noindent
for a function $\cH$ of $x_u$ and $z$ and a function $c(x_u)$ depending only on $x_u$.
\end{prop}

\begin{proof}
The general solution to the first equation in \eqref{eq:generallocalequations} is given by:
\begin{eqnarray*}
\cF = 2\int \frf  \,\dd z + c(x_u)
\end{eqnarray*}

\noindent
for a function $c(x_u)$ of the coordinate $x_u$. Then, since $\kappa$ needs to be chosen as to solve the equations \eqref{eq:exteriorconjugateparallelismtorsion}, the remaining equations are solved by taking:
\begin{eqnarray*}
\kappa_u   = \frac{1}{2} e^{- \cF }(\frf \,\cH - \partial_z \cH ) \, , \qquad \kappa_z  = 0
\end{eqnarray*}

\noindent
and hence we conclude.
\end{proof}

\noindent
Taking the coordinates $(x_u,x_v,z)$ to be the Cartesian coordinates of $\mathbb{R}^3$, the previous theorem allows to construct a plethora of Lorentzian three-manifolds of the form $(\mathbb{R}^3,g)$ that admit skew-torsion parallel spinors. 

\begin{cor}
Lorentzian metrics admitting skew torsion parallel spinors are locally parametrized by one function of two variables and one function of one variable.
\end{cor}


\subsection{The lightlike foliation and the dynamics of $u\in\Omega^1(M)$.}


Let $(M,g,u,\frf)$ skew-torsion parallel spinor. Then, the kernel $u^{\perp_g}\subset TM$ of $u$ defines an integrable distribution in $TM$ whose corresponding foliation we have denoted earlier by $\cX_u \subset M$. This foliation is \emph{lightlike} in the sense that the restriction of the Lorentzian metric $g$ to the tangent spaces of the leaves of $\cX_u$ is degenerate. Equations \eqref{eq:derivativeconjugateparallelismtorsion} immediately imply that the Levi-Civita connection $\nabla^g$ preserves $u^{\perp_g}\subset TM$ and therefore defines by restriction a connection on each of the leaves of $\cX_u$. This connection, which we denote by $D$ momentarily, satisfies:
\begin{eqnarray}
D u^{\sharp_g} = \frf\,n\otimes u^{\sharp_g} \, , \qquad D n =  - \kappa\otimes u^{\sharp_g} 
\end{eqnarray} 

\noindent
where we are using that $u^{\perp_g}\subset TM$ is spanned by $u^{\sharp_g}$ and $n^{\sharp_g}$. A direct computation gives the following formulas for the curvature tensor of $D$:
\begin{eqnarray*}
& R^{D}_{u n} u^{\sharp_g} = D_{u^{\sharp_g}} D_{n^{\sharp_g}} u^{\sharp_g} - D_{n^{\sharp_g}} D_{u^{\sharp_g}} u^{\sharp_g} - D_{[u^{\sharp_g}, n^{\sharp_g}]} u^{\sharp_g}  = 0  \\
& R^{D}_{u n} n^{\sharp_g} = D_{u^{\sharp_g}} D_{n^{\sharp_g}} n^{\sharp_g} - D_{n^{\sharp_g}} D_{u^{\sharp_g}} n^{\sharp_g} - D_{[u^{\sharp_g}, n^{\sharp_g}]} n^{\sharp_g}  = -(\dd\frf(n) + \frf^2) u^{\sharp_g}  
\end{eqnarray*}

\noindent
What we find interesting about this expression is that it does not depend on $\kappa\in \Omega^1(M)$, that is, it does not depend on the one-form that determines the specific form of the differential system that each null coframe $(u,v,n)$ satisfies, but only on the null coframe itself and on the underlying fixed torsion. This allows for a simple criteria that guarantees that the leaves of $\cX_u$ carry a flat connection and therefore an affine structure.

\begin{prop}
Let  $(M,g,u,\frf)$ be a skew-torsion Lorenztian three-manifold such that:
\begin{equation*}
\dd\frf(n) + \frf^2 = 0
\end{equation*}

\noindent
for any, and hence for all, null coframes associated to $u\in \Omega^1(M)$. Then the foliation $\cX_u \subset M$ associated to $u\in \Omega^1(M)$ carries an affine structure induced by the Levi-Civita connection on $\nabla^g$.
\end{prop}

\noindent
The condition given in the previous proposition brings us to one of the cases considered by the authors of \cite{HanounahMehidi} in their study of completeness of foliated structures of compact Lorentzian manifolds, see also \cite{LeistneSchliebnerr,MehidiZeghib}, and supposes an excellent starting point to study the geodesic completeness of compact skew-torsion Lorentzian three-manifolds. Since the latter are in particular equipped with  a nowhere vanishing isotropic Killing vector field, their completeness depends on the dynamics of $u$: if the dynamics of $u$ is non-equicontinuous then $(M,g)$ is geodesically complete, as proven by Zeghib in \cite{Zeghib}, where all such non-equicontinuous flows are classified. On the other hand, the geodesic completeness of a compact Lorentzian manifold equipped with a lightlike Killing vector field with equicontinuous flow seems to be a completely open problem \cite[Page 3]{HanounahMehidi}. 

\begin{prop}
	\label{prop:equicontinuous}
Let $(M,g,\frf,u)$ be a skew-torsion Lorentzian three-manifold. Then, the flow of $u\in \Omega^1(M)$ is equicontinuous if and only if there exists a null coframe $(u,v,n)\in \cP_u$ and a function $F$ satisfying the following differential system:
\begin{equation}
	\label{eq:conditionequi}
	\kappa(u) + \dd F(u) + \frf = 0
\end{equation}

\noindent
where $\kappa\in\Omega^1(M)$ is the one-form respect to which $(u,v,n)$ satisfies \eqref{eq:exteriorconjugateparallelismtorsion}.
\end{prop}

\begin{proof}
By \cite[Lemma 7.4]{HanounahMehidi}, the flow of $u$ is equicontinuous if and only if there exists a null coframe $(u,v,n)\in \cP_u$ that is preserved by $u$. We compute:
\begin{equation*}
\cL_u u = 0\, , \qquad \cL_u v = (\frf + \kappa(u)) n\, , \qquad \cL_u n = - (\frf + \kappa(u)) u
\end{equation*}

\noindent
Every other null coframe satisfies an analogous equation with respect to $\kappa^{\prime}\in \Omega^1(M)$ as prescribed by Equation \eqref{eq:kapparelation}. Choose a new null coframe of the form:
\begin{equation*}
(u,v^{\prime},n^{\prime})	= F\cdot (u,v,n) = (u, v -\frac{F^2}{2} u + F n , n - F u)
\end{equation*}

\noindent
where $F\in C^{\infty}(M)$ satisfies the differential equation:
\begin{equation*}
\kappa(u) + \dd F(u) + \frf = 0
\end{equation*}

\noindent
Then, by Equation \eqref{eq:kapparelation} it follows that $\kappa^{\prime}(u) = -\frf$ and hence $\cL_u u = 0$, $\cL_u v^{\prime} = 0$ and $\cL_u n^{\prime} = 0$, implying that the flow of $u$ preserves a Riemannian metric and is therefore equicontinous. 
\end{proof}

\noindent
Solutions to Equation \eqref{eq:conditionequi} may be obstructed. Consider first the following \emph{weaker} form of Equation \eqref{eq:conditionequi}:
\begin{equation*}
\alpha(u) = -(\kappa(u)   + \frf)
\end{equation*}

\noindent
for a one-form $\alpha\in \Omega^1(M)$. Solutions to this equation always exist: simply set $\alpha = -(\kappa(u)   + \frf) v$. It follows that every solution is of the form:
\begin{equation*}
\alpha = -(\kappa(u)   + \frf) v + \rho_1 u + \rho_2 n
\end{equation*}

\noindent
for functions $\rho_1 , \rho_2 \in C^{\infty}(M)$. Hence, the space of solutions is an affine space $\mathbb{A}_{u,\kappa}$ modelled on the sections of the distribution $u^{\sharp_g}\subset TM$ given by the kernel of $u\in \Omega^1(M)$. 

\begin{cor}
Equation \eqref{eq:conditionequi} admits a solution if and only if $\mathbb{A}_{u,\kappa}$ contains an exact one-form. 
\end{cor}

\noindent
Hence, by the previous corollary it is natural to first impose $\alpha = -(\kappa(u)   + \frf) v + \rho_1 u + \rho_2 n$ to be closed for appropriately chosen functions $\rho_1 , \rho_2 \in C^{\infty}(M)$, and then evaluate if this closed one-form is exact, which would lead to a cohomological obstruction. In particular, we obtain the following criteria for the geodesic completeness of a compact skew-torsion Lorentzian three-manifold $(M,g,u\frf)$.

\begin{cor}
If $\mathbb{A}_{u,\kappa}$ does not contain an exact one-form then $(M,g)$ is geodesically complete.
\end{cor}

\noindent
To end this subsection we recall now the following result, which, as communicated by Hanounah and Mehidi and holds more generally for compact Lorentzian three-manifolds equipped with a nowhere vanishing null Killing vector.

\begin{prop}
\label{prop:completeness}
Let $(M,g,u,\frf)$ be a compact skew-torsion Lorentzian three-manifold that is not diffeomorphic to torus bundle over $S^1$. Then $(M,g)$ is geodesically complete. 
\end{prop} 
 
\begin{proof}
If the flow of $u$ is not equicontinuous then $(M,g)$ is geodesically complete by the results of \cite{Zeghib}. Assume then that the flow of $u$ is equicontinuous. Then, by Proposition \ref{prop:equicontinuous}, there exists a null coframe $(u,v,n)\in \cP_u$ satisfying:
\begin{eqnarray*}
\frac{1}{2} \dd u =\frf  \ast_g u = \frf\,n\wedge u\, , \qquad  \dd v =  \kappa(v) u \wedge n\, , \qquad \dd n = \kappa(n) u\wedge n 
\end{eqnarray*}

\noindent
In particular, the distribution defined by the kernel of $n\in\Omega^1(M)$ is integrable. Denote by $\cX_n \subset  M$ a given leaf of the corresponding foliation, which has vanishing Godbillon-Vey invariant since:
\begin{eqnarray*}
\kappa(n) u \wedge \dd(\kappa(n) u) = \kappa(n)^2 u\wedge \dd u = 0
\end{eqnarray*}

\noindent
The orthogonal projection of the Levi-Civita connection $\nabla^g$ to the distribution determined by the kernel of $n$ defines a connection $D^n$ on $T\cX_n$ that satisfies:
\begin{eqnarray*}
	D^n u^{\sharp_g} = 0\, ,\qquad D^n v^{\sharp_g} = 0\
\end{eqnarray*}

\noindent
and has vanishing torsion. Hence, $D^n$ defines a flat connection  with vanishing torsion on $\cX_n$ that in addition preserves the Riemannian metric $q$ defined on $T\cX_n$ by $q := u\otimes u + v\otimes v$. This implies that $M$ admits a codimension-one foliation with flat leaves, that can only be planes, cylinders or tori. Furthermore, we have:
\begin{eqnarray*}
[u^{\sharp_g},v^{\sharp_g}] = \nabla^g_{u^{\sharp_g}}v^{\sharp_g} - \nabla^g_{v^{\sharp_g}}u^{\sharp_g} = (\frf + \kappa(u)) n^{\sharp_g} = 0
\end{eqnarray*}

\noindent
and therefore, $(u^{\sharp_g},v^{\sharp_g})$ defines a non-singular action of $\mathbb{R}^2$ on $M$ whose leaves are precisely the leaves $\cX_n\subset M$ of the foliation defined by the kernel of $n$. Hence, a classical result in the theory of rank-two compact three-manifolds implies that $M$ is diffeomorphic to a torus bundle over $S^1$ and hence we conclude \cite{FoliationsArraut,FoliationsChatelet,FoliationsRosenberg} given by a suspension of $T^2$ over $\mathbb{R}$ by an element in $\mathrm{Sl}(2,\mathbb{Z})$.
\end{proof}

\noindent
The class of examples in \ref{ep:exampleGodbillonVey} gives numerous instances of compact skew-torsion Lorentzian three-manifolds that are not diffeomorphic to a torus bundle over $S^1$ and that are geodesically complete, as it can be independently verified by noticing that they are compact quotients of AdS$_3$, which is a geodesically complete Lorentzian three-manifold. As a corollary to the previous theorem we obtain the following result. Note that, arguing as in the proof of the previous proposition, it follows that if $(M,g,u,\frf)$ is a compact skew-torsion Lorentzian three-manifold with equicontinuous $u$ then $M$ is a torus bundle over $S^1$.


\section{Supersymmetric solutions in NS-NS supergravity}
\label{sec:susygerbes}


In this section we consider the type of skew-torsion parallel spinors that occur as supersymmetric configurations and solutions of NS-NS supergravity. This requires introducing the notion of a \emph{abelian gerbe}, which we proceed to briefly review together with other geometric preliminaries needed to establish NS-NS supergravity together with its three-dimensional Killing spinor equations. 


\subsection{The NS-NS system}
\label{subsec:NSNSsolutions}


In the following $\cC$ will denote a bundle gerbe $\cC := (\cP,\cA,Y)$ \cite{Murray,BehrendXu}, with underlying submersion $Y\to M$, equipped with a fixed connective structure $\cA$ defined on a fixed three-dimensional manifold $M$, and $\cX$ will denote a principal $\mathbb{Z}$ defined on $M$. Given a curving $b\in \Omega^2(Y)$ on $\cC$ we denote its curvature by $H_b\in \Omega^3(M)$. A given equivariant function $\phi \in \cC^{\infty}(\cX)$ defined on $\cX$ does not necessarily descend to $M$. Nonetheless, its exterior derivative does descend to $M$ and defines the \emph{curvature} $\varphi_{\phi}\in \Omega^1(M)$ of $\phi \in \cC^{\infty}(\cX)$. A pair $(\cC,\cX)$ determines a unique NS-NS supergravity on $M$, as prescribed in the following definition.

\begin{definition}
\label{label:NSNSsystem}
The \emph{NS-NS supergravity system}, or \emph{NS-NS system} for short, determined by $(\cC,\cX)$ on $M$ is the following differential system \cite{Ortin,Tomasiello}:
\begin{equation}
\label{eq:NSNSsystem}
\Ric^g + \nabla^g \varphi_{\phi} - \frac{1}{2} H_b \circ_g H_b = 0\, , \qquad \nabla^{g\ast}H_b + \varphi_{\phi}\lrcorner_g H_b = 0\, , \qquad    \nabla^{g\ast}\varphi_{\phi} +  \vert \varphi_{\phi}\vert_g^2 = \vert H_b\vert^2_g 
\end{equation}
	
\noindent
for triples $(g,b,\phi)$ consisting of a Lorentzian metric $g$ on $M$, a curving $b\in \Omega^2(Y)$ on $\cC$, and an equivariant function $\phi\in C^{\infty}(\cX)$, where $H_b\in \Omega^3(M)$ and $\varphi_{\phi}\in \Omega^1(M)$ are the curvatures of $b$ and $\phi$, respectively.
\end{definition}

\begin{remark}
Solutions $(g,b,\phi)$ to equations \eqref{eq:NSNSsystem} are \emph{NS-NS solutions}. The first equation in \eqref{eq:NSNSsystem} is the so-called \emph{Einstein equation}, whereas the second equation in \eqref{eq:NSNSsystem} is called the \emph{Maxwell equation} and the third equation in \eqref{eq:NSNSsystem} is referred to as the \emph{dilaton equation} in the literature. In this context, an equivariant function on $\cX$ corresponds to the \emph{dilaton} of the theory, and therefore we will refer to equivariant functions on $\cX$ as \emph{dilatons}.
\end{remark}

\noindent
Given a NS-NS solution $(g,b,\phi)$, we will refer to the cohomology class $\sigma = [\varphi_{\phi}] \in H^1(M,\mathbb{R})$ determined by $\phi\in \cC^{\infty}(\cX)$ as the \emph{Lee class} of $(g,b,\phi)$. Clearly, different dilatons on $\cX$ define the same Lee class through their curvature. Given $(\cC,\cX)$, we will denote by $\Conf(\cC,\cX)$ the \emph{configuration space} of the NS-NS system defined on $(\cC,\cX)$, namely the set of triples $(g,b,\phi)$ consisting of a Lorentzian metric $g$ on $M$, a curving $b$ on $\cC$ and an equivariant function $\phi$ on $\cX$. Similarly, we will denote by $\Sol(\cC,\cX)\subset \Conf(\cC,\cX)$ the solution set of NS-NS supergravity on $(\cC,\cX)$. The notion of NS-NS system that we have introduced in Definition \ref{label:NSNSsystem} is valid in any dimension. Since in our case $M$ is three-dimensional, we can simplify the NS-NS system accordingly. Given a curving $b\in \Omega^2(Y)$ with curvature $H_b \in \Omega^3(M)$ we set $H_b = - 2 \frf_b \nu_g$ for a unique function $\frf_b \in C^{\infty}(M)$ determined by:
\begin{equation*}
\frf_b = \frac{1}{2}\ast_g H_b
\end{equation*}

\noindent
A quick computation gives the following formulae:
\begin{equation*}
H_b\circ_g H_b = -4\frf_b^2\, g\, , \qquad \vert H_b\vert^2_g  = - 4\frf_b^2\, , \qquad \varphi_{\phi}\lrcorner_g H_b = -2\frf_b \ast_g \varphi_{\phi}\, , \qquad \nabla^{g\ast} H_b = 2 \ast_g \dd\frf_b 
\end{equation*}

\noindent
Hence, the three-dimensional NS-NS system is equivalent to the following system of equations:
\begin{equation}
\label{eq:NSNSsystem3d}
\Ric^g + \nabla^g \varphi_{\phi} + 2 \frf_b^2 g= 0\, , \qquad \dd \frf_b = \frf_b \varphi_{\phi}\, , \qquad    \nabla^{g\ast}\varphi_{\phi} +  \vert \varphi_{\phi}\vert_g^2 + 4\frf_b^2 = 0
\end{equation}

\noindent
for triples $(g,b,\phi)\in \Conf(\cC,\cX)$. 

\begin{remark}
Left-invariant solutions on three-dimensional Lie groups to the first two equations in the NS-NS system \eqref{eq:NSNSsystem3d} have been studied in \cite{CortesKrusche} in the context of generalized geometry, since they appear naturally as the conditions for a generalized metric to be generalized Ricci flat \cite{FernandezStreets}. Hence, if also satisfying the dilaton equation in \eqref{eq:NSNSsystem3d}, these solutions give examples of NS-NS supergravity solutions.
\end{remark}

\begin{definition}
A triple $(g,b,\phi)\in \Conf(\cC,\cX)$ is \emph{flux-less} if $b$ is flat, namely if $\frf_b=0$ identically on $M$, and is \emph{flux} otherwise. A triple $(g,b,\phi)\in \Conf(\cC,\cX)$ is \emph{trivial} if its flux-less and $g$ is flat. 
\end{definition}
  
\noindent
For a trivial triple $(g,b,\phi)\in \Sol(\cC,\cX)$ the NS-NS system reduces to:
\begin{equation*}
\nabla^g\varphi_{\phi} = 0\, , \qquad \vert\varphi_{\phi}\vert_g^2 = 0
\end{equation*}

\noindent
and thus it follows that $\varphi_{\phi}$ is a parallel light-like one-form on $M$ and consequently $(M,g,\varphi_{\phi})$ defines a three-dimensional flat Brinkmann space-time \cite{MehidiZeghib}. Note that this class of space-times can be very \emph{non-trivial} as Lorentzian manifolds \cite{Carriere}.

\begin{prop}
If $(g,b,\varphi)\in \Sol(\cC,M)$ non-trivial. Then $\nabla^g\varphi_{\phi}  \neq 0$ at some point in $M$.
\end{prop}

\begin{proof}
Let $(g,b,\phi)\in \Sol(\cC,\cX)$ be non-trivial and assume that $\nabla^g\varphi_{\phi} =0$ identically. The norm of $\varphi_{\phi}$ is then constant, which by the third equation in \eqref{eq:NSNSsystem3d} implies that $\frf_b$ is also constant. Hence, either $\frf_b \neq 0$ or $\frf_b = 0$ identically on $M$. If $\frf_b \neq 0$, the second equation in \eqref{eq:NSNSsystem3d} implies $\varphi_{\phi} = 0$, which plugged back into the third equation of \eqref{eq:NSNSsystem3d} gives a contradiction. Therefore, $\frf_b = 0$ and the NS-NS system reduces to:
\begin{equation*}
	\Ric^g  = 0\, ,   \qquad      \vert \varphi_{\phi} \vert_g^2  = 0
\end{equation*}

\noindent
We conclude that $(g,b,\phi)\in \Sol(\cC,\cX)$ defines a three-dimensional Ricci-flat Brinkmann space-time which is therefore flat and thus $(g,b,\phi)$ is trivial.  
\end{proof}

\noindent
In the following we will be interested in non-trivial NS-NS solutions.

\begin{lemma}
\label{lemma:fluxsolution}
Let $(g,b,\phi)\in \Sol(\cC,\cX)$. Then, either $\frf_b = 0$ or else:
\begin{equation*}
\varphi_{\phi} = \dd \phi\, , \qquad \frf_b = c e^{\phi}\, , \qquad c\in \mathbb{R}^{\ast}
\end{equation*}

\noindent
for a function $\phi\in C^{\infty}(M)$ unique modulo an additive constant. In particular $\frf_b\in C^{\infty}(M)$ is nowhere vanishing.
\end{lemma}

\begin{proof}
Let $(g,b,\varphi_{\phi})\in \Sol(\cC,\cX)$. If $\frf_b$ is nowhere vanishing, then the second equation in \eqref{eq:NSNSsystem3d} is equivalent to:
\begin{equation*}
\varphi_{\phi} = \dd\log(\vert \frf_b\vert) 
\end{equation*}

\noindent
and thus there exists a function $\phi\in C^{\infty}(M)$ such that $\varphi_{\phi}  = \dd\phi$ and $\frf_b = c e^{\phi}$ for a non-zero real constant $c$. Suppose now that $\frf_b$ has a zero on $m\in M$. Every smooth path $\gamma\colon \mathbb{R}\to M$ in $M$ passing through $m\in M$ at $t=0$ satisfies:
\begin{equation*}
\partial_t (\frf_b (\gamma_t)) = \frf_b (\gamma_t) \gamma^{\ast}_t\varphi_{\phi} = \frf_b (\gamma_t) F_t
\end{equation*}

\noindent
for a certain real function $F_t$ of $t$. This equation becomes an ordinary differential equation for the real function $\frf_b (\gamma_t)$ which, if it has a zero at $t=0$ then it is identically zero. Since this property holds for every such smooth curve $\gamma$ and $M$ is connected, we conclude that if $\frf_b$ is zero at some point then it is identically zero.
\end{proof}

\noindent
By the previous lemma, the study of NS-NS solutions splits into the study of flux-less NS-NS solutions, namely solutions with flat $b$-field, and \emph{flux} NS-NS solutions, namely solutions which have nowhere vanishing curvature $\frf_b$. For flux-less triples $(g,b,\phi)\in \Conf(\cC,M)$ the NS-NS system reduces to:

\begin{equation*}
\Ric^g + \nabla^g \varphi = 0\, ,   \qquad    \nabla^{g\ast}\varphi +  \vert \varphi\vert_g^2  = 0
\end{equation*}

\noindent
and thus it effectively reduces to differential system for pairs $(g,\varphi)$, where $g$ is a Lorentzian metric on $M$ and $\varphi$ is a closed one-form. Note that solutions to the first equation above correspond to steady Lorentzian Ricci solitons \cite{Gavino}. On the other hand, for triples $(g,b,\varphi)$ with $\frf_b \neq 0$ the NS-NS system reduces to a family of differential systems:
\begin{equation}
\label{eq:NSNSsystem3dreducedphi}
\Ric^g + \nabla^g \dd\phi + 2 c^2  e^{2\phi} g= 0\, ,  \qquad    \nabla^{g\ast}\dd\phi  +  \vert \dd\phi \vert_g^2 +  4 c^2  e^{2\phi} =  0\, , \qquad \frf_b = c e^{\phi}
\end{equation}

\noindent
parametrized by the non-zero real constant $c\in\mathbb{R}^{\ast}$. Hence, the NS-NS system decouples: the first two equations in \eqref{eq:NSNSsystem3dreducedphi} can be \emph{solved} independently and, once a solution $(g,\phi)$ has been found, a curving $b$ solving the third equation can be found as soon as $\frf_b$ defines an integral class in de Rham cohomology.

\begin{prop}
Let $M$ be an oriented compact three-manifold without boundary. Then $M$ admits no flux NS-NS solutions.
\end{prop}

\begin{proof}
Integrate the second equation in \eqref{eq:NSNSsystem3dreducedphi}.
\end{proof}

\noindent
Three-dimensional NS-NS flux solutions can however be globally hyperbolic with compact Cauchy hypersurface, and therefore can be studied from the point of view of the celebrated program initiated by Mess \cite{Mess,BonsanteSeppi}.


\subsection{Supersymmetric configurations}


The notion of \emph{supersymmetric configuration} originated in supergravity and, for the three-dimensional NS-NS supergravity that we consider in this note, it involves skew-torsion parallel spinors. Given an element $(g,b,\phi)\in \Conf(\cC,\cX)$ we denote by $\nabla^{g,b}$ the unique metric connection on $(M,g)$ with totally skew-symmetric torsion given by $H_b\in \Omega^3(M)$, the curvature of $b\in\Omega^2(Y\times_M Y)$. This is the natural way in which supersymmetry realizes geometrically the notion of completely skew torsion of a metric connection. For ease of notation we denote by the same symbol $\nabla^{g,b}$ its lift to any irreducible spinor bundle defined on $(M,g)$.

\begin{definition}
A \emph{supersymmetric configuration} on $(\cC,\cX)$ is triple $(g,b,\phi)$ consisting of a Lorentzian metric $g$ on $M$, an equivariant function $\phi\in \cC^{\infty}(\cX)$, and a curving $b\in \Omega^2(Y)$ satisfying the following differential system:
\begin{equation}
\label{eq:KSE3d}
\nabla^{g,b} \varepsilon = 0\, , \qquad  (\varphi_{\phi} + H_b ) \cdot \varepsilon = 0
\end{equation}

\noindent
for a non-zero spinor $\varepsilon\in\Gamma(S)$ of an irreducible paired spinor bundle $(S,\Gamma,\cB)$ defined on $(M,g)$. Such spinor $\varepsilon\in \Gamma(S)$ is a \emph{supersymmetry parameter} or \emph{supersymmetry generator} for the supersymmetric configuration $(g,\varphi,b)$.
\end{definition}

\noindent
Therefore, the supersymmetry parameter of a supersymmetric configuration is in particular a skew-torsion parallel spinor. We denote by $\Conf_s(\cC,\cX)\subset \Conf(\cC,\cX)$ the set of supersymmetric configurations on $(\cC,\cX)$. 
 
\begin{prop}
\label{prop:susyconfiguration}
A triple $(g,b,\phi)\in \Conf_s(\cC,\cX)$ is a supersymmetric configuration if and only if there exists an isotropic one-form $u\in \Omega^1(M)$ on $(M,g)$ and a function $\frf_b\in C^{\infty}(M)$ such that:
\begin{equation}
\label{eq:susyconditions3d}
\nabla^g u = \frf_b \ast_g u\, , \qquad \varphi_{\phi} = \mathfrak{c}\, u + 2\frf_b n 
\end{equation} 

\noindent
for a function $\mathfrak{c}\in C^{\infty}(M)$, where $H_b = - 2\, \frf_b \nu_g$ is the curvature of the curving $b\in\Omega^2(Y)$.
\end{prop}

\begin{proof}
By Proposition \ref{prop:torsionparallel}, the first equation in \eqref{eq:KSE3d} is satisfied if and only if there exists an isotropic vector field $u\in\mathfrak{X}(M)$ satisfying the first equation in \eqref{eq:susyconditions3d}. On the other hand, by Theorem \ref{thm:GCKS} and Remark \ref{remark:algebraicconstraint}, the second equation in \eqref{eq:KSE3d} is equivalent to:
\begin{eqnarray*}
(\Psi_{\Gamma}^{<})^{-1}(\Gamma\circ\Psi^{-})(\varphi_{\phi} + H_b) \vee_g u=    (\varphi_{\phi} - 2 l \frf_b) \vee_g u = \varphi_{\phi}(u) + l\ast_g(u\wedge \varphi_{\phi}) - 2l\frf_b u
\end{eqnarray*}

\noindent
which is solved precisely by $\varphi_{\phi} = \mathfrak{c}\, u + 2\frf_b n $ for a function $\mathfrak{c}\in C^{\infty}(M)$.
\end{proof}

\noindent
Hence, every supersymmetric configuration on $(\cC,\cX)$ defines a skew-torsion Lorentzian three-manifold $(M,g,u,\frf_b)$. We cannot expect however the converse to be true since, for a skew-torsion Lorentzian three-manifold to be part of a supersymmetric configuration the one-form $( \mathfrak{c}\, u + 2\frf_b n)\in \Omega^1(M)$ needs to be not only closed, but exact.

\begin{prop}
\label{prop:exactness}
A skew-torsion Lorentzian three-manifold $(M,g,u,\frf)$ defines a supersymmetric configuration if and only if there exists a function $\mathfrak{c}\in C^{\infty}(M)$ such that:
\begin{equation*}
[\mathfrak{c}\, u + 2\frf_b n]\in H^1(M,\mathbb{Z})
\end{equation*}

\noindent
defines a integral class.
\end{prop}

\begin{proof}
The \emph{only if} condition follows directly from Proposition \ref{prop:susyconfiguration}. To prove the converse, simply note that since $\mathfrak{c}\, u + 2\frf_b n$ is an integral one-form, it follows that there exists a principal $\mathbb{Z}$ bundle $\cX$ and an equivariant function $\phi \colon \cX\to \mathbb{R}$ such that $\varphi_{\phi} = \mathfrak{c}\, u + 2\frf_b n$.
\end{proof}

\noindent
We end this subsection by computing the covariant derivative of the curvature $\varphi_{\phi}$ of the dilaton.

\begin{lemma}
\label{lemma:nablavarphi}
Let $(g,b,\phi,\varepsilon)\in \Conf_s(\cC,\cX)$ be a supersymmetric configuration with associated Dirac current $u\in \Omega^1(M)$ and let $(u,v,n)\in \cP_u$ be any null coframe. Then:
\begin{eqnarray}
\label{eq:nablavarphi}
\nabla^g\varphi_{\phi} = ( 2 \dd\frf_b(v) - \mathfrak{c}\,\frf_b) u\odot n +  2\frf_b^2 u\odot v + 2 \dd\frf_b(n) n\otimes n + (\dd\mathfrak{c}(v) - 2 \frf_b \,\kappa(v)) u\otimes u
\end{eqnarray}
	
\noindent
where $\varphi_{\phi} = \mathfrak{c}\, u + 2\frf_b n$. Furthermore, $(u,v,n)$ satisfies the differential system \eqref{eq:exteriorconjugateparallelismtorsion} with respect to a one-form $\kappa\in \Omega^1(M)$ that satisfies the following relations:
\begin{eqnarray}
\label{eq:relationsck}
\dd\mathfrak{c}(u) = 2 \frf_b  (\frf_b +  \kappa(u))\, , \qquad \dd\mathfrak{c}(n) = 2 (\dd\frf_b(v) + \frf_b \kappa(n) - \mathfrak{c}\,\frf_b)
\end{eqnarray}
	
\noindent
in terms of $\mathfrak{c}\in C^{\infty}(M)$.
\end{lemma}

\begin{proof}
By Proposition \ref{prop:susyconfiguration} we have $\varphi_{\phi} = \mathfrak{c}\,u +  2\frf_b n$ in any given null coframe $(u,v,n)\in \cP_u$. Since $\varphi$ must be exact, it is necessarily closed, hence:
\begin{eqnarray*}
0 = \dd\varphi_{\phi} = \dd\mathfrak{c}\wedge u + \mathfrak{c}\,\dd u +  2 \dd\frf_b\wedge n + 2 \frf_b \dd n =  \dd\mathfrak{c}\wedge u + \mathfrak{c}\,\frf_b n\wedge u + 2 \dd\frf_b\wedge n  + 2 \frf_b u\wedge (\frf_b v + \kappa)
\end{eqnarray*}
	
\noindent
Solving the previous equation by using the integrability conditions for $(u,v,n)$, we obtain the given relations \eqref{eq:relationsck}. Applying now Equation \eqref{eq:derivativeconjugateparallelismtorsion} to compute the covariant derivative of $n$, we obtain:
\begin{eqnarray*}
&\nabla^g \varphi_{\phi} = \dd\mathfrak{c}\otimes u + \mathfrak{c} \nabla^g u + 2 \dd\frf_b \otimes n + 2\frf_b\nabla^g n \\
& = \dd\mathfrak{c}\otimes u + \mathfrak{c}\, \frf_b (n\otimes u - u \otimes n) + 2 \dd\frf_b \otimes n + 2\frf_b (\frf_b \,  u\otimes v - \kappa\otimes u ) 
\end{eqnarray*} 
	
\noindent
Plugging Equation \eqref{eq:relationsck} into the previous equation and expanding we obtain \eqref{eq:nablavarphi} and we conclude.
\end{proof}

\noindent
We will need the previous lemma to study supersymmetric solutions in the following subsection.
 

\subsection{Supersymmetric solutions}


We consider now the supersymmetric solutions of the NS-NS system, which satisfy the natural curvature condition on skew-torsion Lorentzian three-manifolds that we will consider in this article.
\begin{definition}
A \emph{supersymmetric solution} is a supersymmetric configuration $(g,b,\phi)\in \Conf_s(\cC,\cX)$ that satisfies the NS-NS system on $(\cC,\cX)$.
\end{definition}

\noindent
We denote by $\Sol_s(\cC,\cX)$ the set of supersymmetric solutions on $(\cC,\cX)$. By the results of Section \ref{subsec:NSNSsolutions}, we have to distinguish between flux-less and flux supersymmetric NS-NS solutions. A quick inspection of equations \eqref{eq:NSNSsystem3d} and \eqref{eq:susyconditions3d} reveals that flux-less supersymmetric NS-NS solutions consist of flat bundle gerbes over flat Brinkmann space-times equipped with a parallel isotropic vector field given by the curvature of the dilaton. We focus then on flux solutions.

\begin{prop}
\label{prop:ddf}
A supersymmetric configuration $(g,b,\phi,\varepsilon)\in \Conf_s(\cC,\cX)$ satisfies the Einstein equation of the NS-NS system if and only if:  
\begin{eqnarray}
\label{eq:ddf}
& \dd\frf_b (v) = \mathfrak{c}\,\frf_b\, , \qquad \dd\frf_b(n) = 2 \frf_b^2 \, , \qquad \dd\mathfrak{c}(v) = \mathfrak{l} +   \frf_b \,\kappa(v) 
\end{eqnarray}

\noindent
where $\varphi_{\phi} = \mathfrak{c}\, u + 2\frf_b n$ in terms of any null coframe $(u,v,n)\in \cP_u$ associated to the Dirac current $u$ of $\varepsilon\in \Gamma(S)$.
\end{prop}

\begin{proof}
Let $(g,b,\phi,\varepsilon)$ be a supersymmetric configuration with associated Dirac current $u\in \Omega^1(M)$ and fix a null coframe $(u,v,n)\in \cP_u$. Then, by Proposition \ref{prop:susyconfiguration} we have $\varphi_{\phi} = \dd\phi = \mathfrak{c}\, u + 2\frf_b n$ in terms of a function $\mathfrak{c}\in C^{\infty}(M)$. Using equations \eqref{eq:Riccicurvature} and \eqref{eq:nablavarphi}, we compute:
\begin{eqnarray}
	& 0 = \mathrm{Ric}^g + \nabla^g\varphi_{\phi} + 2\frf_b^2 = - \dd\frf(v) u\odot n - \dd\frf(n) g - \dd\frf(n)\, n\otimes n + (\frf\, \kappa(v) - \mathfrak{l})   u\otimes u \nonumber \\
	& + ( 2 \dd\frf_b(v) - \mathfrak{c}\,\frf_b) u\odot n +  2\frf_b^2 u\odot v + 2 \dd\frf_b(n) n\otimes n + (\dd\mathfrak{c}(v) - 2 \frf_b \,\kappa(v)) u\otimes u \label{eq:computationEinstein}\\
	& =   (\dd\mathfrak{c}(v) - \mathfrak{l} -   \frf_b \,\kappa(v))   u\otimes u  + (\dd\frf_b(v) - \mathfrak{c}\,\frf_b) u\odot n +  (2\frf_b^2- \dd\frf(n)) u\odot v  \nonumber
\end{eqnarray}

\noindent
Hence $\dd\frf_b (v) = \mathfrak{c}\,\frf_b$, $\dd\frf_b(n) = 2 \frf_b^2$, $\dd\mathfrak{c}(v) = \mathfrak{l} +   \frf_b \,\kappa(v)$ which together with Equation \eqref{eq:relationsck} we conclude.
\end{proof}

\noindent
By the previous proposition, we introduce the following notation:
\begin{equation*}
\frc^v_b := v(\log(\frf_b)) = \dd\log(\frf_b)(v)
\end{equation*}

\noindent
that we will use in the following for simplicity in the exposition. Let $(g,b,\phi)\in \Sol(\cC,\cX)$ be a NS-NS solution. Then, by the third equation in \eqref{eq:NSNSsystem3dreducedphi}, we can equivalently write $(g,b,\phi) = (g,b,\log(\frf_b)+c)$ for a certain constant $c\in \mathbb{R}$ which is uninteresting to our purposes. Hence, the pair $(g,b)$ already determines $(g,b,\phi)$ modulo the aforementioned constant, so in the following we will occasionally refer to a NS-NS solution simply by $(g,b)\in \Sol(\cX,\cQ)$.

\begin{prop}
A tuple $(g,b,\phi,\varepsilon)$ is a supersymmetric solution if and only if $(M,g,\varepsilon)$ is a skew-torsion Lorentzian manifold with respect to a nowhere vanishing $\frf_b$, the one-form:
\begin{equation*}
\dd\log(\frf_b)(v)\, u + 2\frf_b n \in \Omega^1(M) 
\end{equation*}

\noindent
is exact and $(g,b,\log(\frf_b))$ satisfies the Einstein equation of the NS-NS system. 
\end{prop}

\begin{proof}
The \emph{only if} direction follows directly from equations \eqref{eq:NSNSsystem3dreducedphi} together with propositions \ref{prop:susyconfiguration}, \ref{prop:exactness} and \ref{prop:ddf}. To prove the \emph{if} direction we consider a skew-torsion Lorentzian three-manifold $(M,g,\frf_b,u)$ such that $\dd\log(\frf_b)(v)\, u + 2\frf_b n \in \Omega^1(M)$ is exact and $(g,b,\log(\frf_b))$ satisfies the Einstein equation in \eqref{eq:NSNSsystem3d}. Since the one-form $\dd \log(\frf_b)(v)\, u + 2\frf_b n \in \Omega^1(M)$ is exact, there exists a function $\phi\in C^{\infty}(M)$ such that:
\begin{eqnarray*}
\dd \phi = \dd \log(\frf_b)(v)\, u + 2\frf_b n
\end{eqnarray*}

\noindent
Evaluating on $(u,v,n)$ we obtain:
\begin{eqnarray*}
\dd f (u) = 0\, , \qquad \dd f(v) = \dd \log(\frf_b)(v)\, , \qquad \dd f(n) = 2\frf_b (n) = \dd\log(\frf_b)(n)
\end{eqnarray*}

\noindent
where we have used Proposition \ref{prop:ddf} by virtue of the Einstein equation being satisfied by $(g,b,\log(\frf_b))$. Since $\dd\frf_b(u) = 0$ by the integrability conditions \eqref{eq:integrabilityconditions} of $(u,v,n)$, we obtain $\dd \phi = \dd\log(\frf_b)$ and consequently $\frf_b = c e^{\phi}$ for a non-zero constant $c\in \mathbb{R}$. Hence, setting $\varphi_{\phi} =\dd\phi = \dd\log(\frf_b)(v)\, u + 2\frf_b n$, the tuple $(g,b,\phi,\varepsilon)$ becomes a supersymmetric configuration that satisfies the Einstein and Maxwell equations of the NS-NS system. Hence, only the dilaton equation needs to be checked. We compute:
\begin{eqnarray*}
0 = \nabla^{g\ast}\varphi_{\phi} + \vert\varphi_{\phi}\vert^2_g +4\frf_b^2 = - 4\frf_b^2 - 2\dd\frf_b(n) + 4\frf_b^2 + 4\frf_b^2 = 2(2\frf_b^2 - \dd\frf(n))
\end{eqnarray*}

\noindent
and therefore we conclude since by Proposition \ref{prop:ddf} we have $2\frf_b^2 = \dd\frf(n)$.
\end{proof}

\begin{thm}
\label{thm:susysolutions}
A NS-NS flux configuration $(g,b)\in \Conf(\cC,\cX)$ is a supersymmetric solution of the NS-NS system if and only if there exists a nowhere vanishing isotropic vector $u\in\Omega^1(M)$ satisfying:    
\begin{align}
& \dd u = 2 \frf_b\,n\wedge u\, , \qquad  2\frf_b \dd v = (\mathfrak{K}\, u + u(\frc_b^v) v)\wedge n\, , \qquad 2\frf_b\dd n =  u\wedge (u(\frc_b^v) v + n(\frc_b^v) n) \label{eq:susysystem} \\
& \frac{n(\mathfrak{K})}{2\frf_b}   =  \frac{u(\frc_b^v)}{4\frf_b^2} \mathfrak{K} +  v(\frac{n(\frc_b^v)}{2\frf_b}) + \frac{n(\frc_b^v)^2}{4\frf^2_b} - v(\frc_b^v)\, , \qquad \dd\frf_b(n) = 2 \frf_b^2 \, , \qquad  0 = [\frc_b^v u + 2\frf_b n] \in H^1(M,\mathbb{R})\nonumber 
\end{align}

\noindent
for a function $\mathfrak{K}\in C^{\infty}(M)$ and for any, and hence for all, null coframes $(u,v,n) \in \cP_u$ associated to $u\in \Omega^1(M)$.
\end{thm}

\begin{remark}
Supersymmetric configurations and solutions of ten-dimensional NS-NS supergravity have been studied before in the supergravity literature under the assumption of existence of a compatible absolute parallelism \cite{Figueroa-OFarrill:2003fkz,Figueroa-OFarrill:2003gow,Kawano:2003af}. In our framework, this case requires setting $\kappa = \frf_b v$, which imposes strong constraints on the underlying geometry. 
\end{remark}

\begin{remark}
By Remark \ref{remark:Kundttorsion}, we can refer to three-dimensional NS-NS supersymmetric solutions, and in particular to the Lorentzian three-manifolds characterized in the previous theorem as \emph{supersymmetric Kundt three-manifolds}. This defines an interesting class of Kundt three-manifolds that has not been studied in the mathematics literature.
\end{remark}

\begin{proof}
Let $(g,b)\in \Sol_s(\cC,\cX)$ be a supersymmetric NS-NS solution with supersymmetry parameter $\varepsilon\in\Gamma(S)$ and associated Dirac current $u\in\Omega^1(M)$. Then, by Lemma  \ref{lemma:conjugatecovarianttorsion} and Proposition \ref{prop:susyconfiguration}, every null coframe $(u,v,n)\in \cP_u$ associated to $u$ satisfies equations \eqref{eq:derivativeconjugateparallelismtorsion} for a certain $\kappa \in \Omega^1(M)$. Equations \eqref{eq:derivativeconjugateparallelismtorsion} imply in turn that $(u,v,n)\in\cP_u$ satisfies the exterior differential system \eqref{eq:exteriorconjugateparallelismtorsion} for the same $\kappa \in \Omega^1(M)$. On the other hand, by Proposition \ref{prop:ddf} we have:
\begin{equation*}
\varphi_{\phi} = \dd\phi = v(\log(\frf_b)) u + 2\frf_b
\end{equation*}

\noindent
as well as $\frf_b = c e^{\phi}$ for a non-zero real constant in $c$ by Lemma \ref{lemma:fluxsolution}. Furthermore, again by Proposition \ref{prop:ddf} together with Lemma \ref{lemma:nablavarphi}, see equations \eqref{eq:relationsck}, we have:
\begin{equation*}
 \mathfrak{c}^v_b = v(\log(\frf_b))\, ,  \quad v(\mathfrak{c}^v_b) = \mathfrak{l} +  \frf_b \,\kappa(v)\, , \quad u(\mathfrak{c}^v_b) = 2 \frf_b  (\frf_b +  \kappa(u))\, , \quad n(\mathfrak{c}^v_b) = 2 \frf_b \kappa(n) \, , \quad 2\frf_b^2 = \dd\frf(n)
\end{equation*}

\noindent
which are satisfied by any given null coframe $(u,v,n)\in\cP_u$ associated to the Dirac current $u\in \Omega^1(M)$ corresponding to the underlying skew-torsion parallel spinor.  These equations imply that the one-form $\kappa$ respect to which $(u,v,n)\in\cP_u$ satisfies the exterior differential system \eqref{eq:exteriorconjugateparallelismtorsion} is given by:
\begin{eqnarray}
\label{eq:equationsproof}
\kappa = \kappa(v) u + (\frac{u(\mathfrak{c}^v_b)}{2\frf_b} - \frf_b) v + \frac{n(\mathfrak{c}^v_b)}{2\frf_b} n = \kappa(v) u + (\frac{u(v(\frf_b))}{2\frf^2_b} - \frf_b) v + (\frac{n(v(\frf_b))}{2\frf_b^2} - \frac{v(\frf_b)}{\frf_b}) n
\end{eqnarray}

\noindent
From this equation, we compute:
\begin{eqnarray*}
& \dd\kappa = (v(\kappa(u)) + \kappa(n) (\frf_b + \kappa(u)) -u(\kappa(v)) ) u\wedge v \\
& + (\kappa(u)\kappa(v) + v(\kappa(n)) + \kappa(n)^2 -n(\kappa(v)) - 2\frf_b \kappa(v)) u\wedge n + (-n(\kappa(u)) + \kappa(u)(\frf_b + \kappa(u)) + u(\kappa(n))) v\wedge n
\end{eqnarray*}

\noindent
Comparing this expression for $\dd\kappa \in \Omega^2(M)$ with the expression obtained in Proposition \ref{prop:dkappa} via the integrability conditions, we obtain the following differential system:

\begin{eqnarray*}
&  v(\kappa(u)) + \kappa(n) (\frf_b + \kappa(u)) -u(\kappa(v)) + v(\frf_b) = 0\\
&  \kappa(u)\kappa(v) + v(\kappa(n)) + \kappa(n)^2 -n(\kappa(v)) - 2\frf_b \kappa(v) = \mathfrak{l}\\
&  \kappa(u)^2 -n(\kappa(u)) + u(\kappa(n)) = 3\frf_b^2
\end{eqnarray*}

\noindent
where:
\begin{eqnarray*}
\kappa(u) := \frac{u(v(\frf_b))}{2\frf^2_b} - \frf_b \, , \qquad \kappa(n) =\frac{n(v(\frf_b))}{2\frf_b^2} - \frac{v(\frf_b)}{\frf_b} 
\end{eqnarray*}

\noindent
Note that the only unknown in the previous system of equations is $\kappa(v) \in C^{\infty}(M)$. Of this system of equations we are interested in the equation that determines $\mathfrak{l}$, which plugged back into the second equation in \eqref{eq:equationsproof} gives the explicit expression for the last equation that we need to consider, which is at this point equivalent to the Einstein equation in the NS-NS system, namely:
\begin{eqnarray*}
v(\mathfrak{c}^v_b) = \mathfrak{l} +  \frf_b \,\kappa(v) = -n(\kappa(v)) - 2\frf_b \kappa(v) + \frac{u(\mathfrak{c}^v_b)}{2\frf_b} \kappa(v) + v(\frac{n(\mathfrak{c}^v_b)}{2\frf_b}) + (\frac{n(\mathfrak{c}^v_b)}{2\frf_b})^2
\end{eqnarray*}

\noindent
This is the equation in the statement of the theorem after setting $\mathfrak{K} := 2\frf_b \kappa(v)$.
 
Conversely, let $(g,b,\phi)\in \Conf(\cC,\cX)$ and assume that we have a nowhere vanishing isotropic vector field $u\in \Omega^1(M)$ on $(M,g)$. Let $(u,v,n)\in \cP_u$ be a null coframe that satisfies all equations in \eqref{eq:susysystem} with respect to a given function $\mathfrak{K}\in C^{\infty}(M)$. By assumption, $(u,v,n)$ satisfies \eqref{eq:exteriorconjugateparallelismtorsion} with respect to a one-form $\kappa \in \Omega^1(M)$ of the form:
\begin{eqnarray*}
\kappa = \frac{\mathfrak{K}}{2\frf_b} u + (\frac{u(\mathfrak{c}^v_b)}{2\frf_b} - \frf_b) v + \frac{n(\mathfrak{c}^v_b)}{2\frf_b} n  
\end{eqnarray*}

\noindent
This in turn implies that $(u,v,n)$ satisfies equations \eqref{eq:derivativeconjugateparallelismtorsion} with respect to the metric $g = u\odot v + n\otimes n$ and the same one-form $\kappa \in \Omega^1(M)$. Hence $(M,g)$ admits a skew-torsion parallel spinor $\varepsilon \in \Gamma(S)$ with respect to a metric connection with totally skew-symmetric torsion $H =  - 2\frf_b \nu_g$. Furthermore, since $0 = [\mathfrak{c}^v_b\, u + 2\frf_b n]\in H^1(M,\mathbb{R})$ there exists a function $\phi\in C^{\infty}(M)$ such that:
\begin{eqnarray*}
	\dd\phi = \mathfrak{c}^v_b\, u + 2\frf_b n
\end{eqnarray*}

\noindent
Hence we obtain that $\frf_b = c e^{\phi}$ for a non-zero real constant $c\in \mathbb{R}$ upon use of $\dd\frf(u) = 0$ by the integrability conditions of $(u,v,n)$. Then, by Proposition \ref{prop:susyconfiguration}, the tuple $(g,\frf,\phi,\varepsilon)$ is in fact a supersymmetric configuration of NS-NS supergravity. Furthermore, condition $2\frf_b^2 = \dd\frf(n)$ implies that the dilaton equation holds by the same computation as in the proof of Proposition \ref{prop:ddf}. On the other hand, by the same computation as in \eqref{eq:computationEinstein}, the Einstein equation evaluated on $(g,\frf_b,\phi)$ reduces to:
\begin{eqnarray*}
v(\mathfrak{c}_b^v) = \mathfrak{l} +   \frf_b \,k
\end{eqnarray*}

\noindent
After identifying $\mathfrak{l}$ in terms of $\frf_b$ and its derivatives with respect to the given parallelism $(u,v,n)\in \cP_u$ this equation reduces to the first equation in the second line of \eqref{eq:susysystem} and therefore is satisfied by assumption.
\end{proof}

\begin{remark}
It is remarkable that if $(u,v,n)\in \cP_u$ satisfies the differential system of the statement of the previous theorem, then any other null coframe $(u,v^{\prime},n^{\prime})\in \cP_u$ also satisfies it, perhaps for a differente $\mathfrak{K}$. In particular, for the curvature of the dilaton, we have:
\begin{eqnarray*}
& \dd \phi = \dd\log(\frf_b)(v^{\prime})\, u + 2\frf_b n^{\prime} = \dd\log(\frf_b)(v - \frac{F^2}{2} u + F n)\, u + 2\frf_b n - 2\frf_b F\, u \\
& = \dd\log(\frf_b)(v)\, u + 2\frf_b n + F \dd\log(\frf_b)(n)\, u - 2\frf_b F\, u = \dd\log(\frf_b)(v)\, u + 2\frf_b n 
\end{eqnarray*}

\noindent
where $(u,v^{\prime},n^{\prime}) = (u,v - \frac{F^2}{2} u + F u,n  - F u)$ for a function $F\in C^{\infty}(M)$.
\end{remark}

\noindent
To the best of our knowledge the previous theorem provides what seems to be one of the first global characterizations of the supersymmetric NS-NS solutions on manifolds of general topology, and as such can be further refined for different particular classes of supersymmetric solutions, for instance, static or stationary supersymmetric solutions, which we plan consider elsewhere. If in Theorem \ref{thm:susysolutions} we are not interested in the relation between the given supersymmetry parameter $\varepsilon\in \Gamma(S)$ and the associated Dirac current and null coframes, then we can reformulate it as follows, which is sometimes more convenient for applications. 

\begin{cor}
A three-manifold $M$ admits a supersymmetric NS-NS solution if and only if it admits a global coframe $(u,v,n)$ satisfying the following differential system    
	\begin{align}
		& \dd u = 2 \frf_b\,n\wedge u\, , \qquad  2\frf_b \dd v = (\mathfrak{K}\, u + u(\frc_b^v) v)\wedge n\, , \qquad 2\frf_b\dd n =  u\wedge (u(\frc_b^v) v + n(\frc_b^v) n) \label{eq:susysystemii} \\
		& \frac{n(\mathfrak{K})}{2\frf_b}   =  \frac{u(\frc_b^v)}{4\frf_b^2} \mathfrak{K} +  v(\frac{n(\frc_b^v)}{2\frf_b}) + \frac{n(\frc_b^v)^2}{4\frf^2_b} - v(\frc_b^v)\, , \qquad \dd\frf_b(n) = 2 \frf_b^2 \, , \qquad  0 = [\frc_b^v u + 2\frf_b n] \in H^1(M,\mathbb{R})\nonumber 
	\end{align}
	
\noindent
for a function $\mathfrak{K}\in C^{\infty}(M)$ and a curving $b$ on a bundle gerbe over $M$.
\end{cor}

\begin{remark}
In a supergravity context, the supersymmetry parameter is many times irrelevant in itself, only the geometric and topological consequences of its existence being important. In this situation, the previous corollary captures the if and only if conditions for a supersymmetry solution to exist with no mention of the underlying Lorentzian metric or supersymmetry parameter. This seems to realize, at least to some extent, the motivation and ideology explained in \cite{Tomasiello:2011eb}, where the \emph{complete} Type-II theory, of which NS-NS supergravity is a subsector, is considered in ten dimensions.
\end{remark} 

\noindent
To end this section, we apply the previous theorem to completely characterize three-dimensional NS-NS supersymmetric solutions locally. By Proposition \ref{prop:localskewtorsion}, given a supersymmetric NS-NS solution $(g,b)$ on $M$ we have local coordinates $(x_u,x_v,z)$ in which te metric and the Dirac current $u\in \Omega^1(M)$ read:
\begin{eqnarray*}
	g = \cH\, \dd x_u\otimes \dd x_u + e^{2\int \frf\,\dd z + c(x_u)} \dd x_u\odot \dd x_v + \dd z \otimes \dd z\, , \quad u = e^{2\int \frf\,\dd z + c(x_u)} \dd x_u
\end{eqnarray*}

\noindent
for a function $\cH$ of $x_u$ and $z$ and a function $c(x_u)$ of $x_u$. In particular, we can choose a null coframe 
$(u,v,n)\in\cP_u$ of the form:
\begin{eqnarray*}
(u = e^{2\int \frf\,\dd z + c(x_u)} \dd x_u ,v = \frac{1}{2}\cH e^{-2\int \frf\,\dd z - c(x_u)} \dd x_u + \dd x_v, n =  \dd z ) \in \cP_u
\end{eqnarray*}

\noindent
We compute:
\begin{eqnarray*}
\dd u = 2\frf_b n\wedge u\, , \qquad \dd v = \frac{1}{2}\partial_z (\cH e^{-2\int \frf\,\dd z - c(x_u)})\, e^{-2\int \frf\,\dd z - c(x_u)} n\wedge u \, , \qquad \dd n = 0
\end{eqnarray*}

\noindent
and hence we conclude that:
\begin{equation*}
\mathfrak{H} = -\frf_b \partial_z (\cH e^{-2\int \frf\,\dd z - c(x_u)})\, e^{-2\int \frf\,\dd z - c(x_u)}\, , \qquad \partial_z \frf_b = 2 \frf_b^2\, , \qquad \partial_{x_v}\frc_b^v = 0\, , \qquad \partial_z\frc_b^v = 0   
\end{equation*}

\noindent
The first equation above determines explicitly $\mathfrak{H}$ in terms of $b$, $\cH$ and $c(x_u)$. Since we must have $\partial_{x_v}\frf_b = 0$ by the integrability conditions of \eqref{eq:susysystem}, the general solution to the second equation above is:
\begin{eqnarray}
\label{eq:localf}
\frf_b = -\frac{1}{2 z + l(x_u)}
\end{eqnarray}

\noindent
for a function $l(x_u)$ of the coordinate $x_u$. The remaining conditions $\partial_{x_v}\frc_b^v = 0$ and $\partial_z\frc_b^v = 0$ in the above equation are respectively equivalent to:
\begin{equation*}
\partial_{x_u}(e^{-2\int \frf\,\dd z - c(x_u)} \frac{\partial_{x_u}\frf_b}{\frf_b}) = 0\, , \qquad \partial_{z}(e^{-2\int \frf\,\dd z - c(x_u)} \frac{\partial_{x_u}\frf_b}{\frf_b}) = 0
\end{equation*}

\noindent
which in turn reduce to:
\begin{eqnarray*}
 \frac{\partial_{x_u}\frf_b}{\frf_b}  = - a\, e^{2\int \frf\,\dd z + c(x_u)}
\end{eqnarray*}
 
\noindent
for a real constant $a\in \mathbb{R}$. Substituting \eqref{eq:localf} into the previous expression, we obtain:
\begin{eqnarray*}
\partial_{x_u} l(x_u) =   a\, e^{c(x_u)}
\end{eqnarray*}

\noindent
The general solution to this equation is $l(x_u) = d + a\int e^{c(x_u)}$ for real constants $a , d \in  \mathbb{R}$. Hence, after relabeling the arbitrary functions of $x_u$ contained in the solution, we obtain:
\begin{eqnarray*}
\frf_b = -\frac{1}{2 z + a(x_u)}\, , \qquad \dd\phi = - \frac{2 \dd z + \partial_{x_u}a(x_u) \dd x_u }{2 z + a(x_u)}    
\end{eqnarray*}

\noindent
as well as:
\begin{eqnarray*}
u =   \frac{\partial_{x_u} a(x)}{2 z + a(x_u)}\dd x_u\, , \quad v = \frac{(a(x_u) + 2 z)\, \cH}{2\partial_{x_u} a(x_u)}   \dd x_u + \dd x_v\, , \quad n =  \dd z\, , \quad (\partial_{x_u} a(x_u))^{2}\mathfrak{K} =  \partial_{z} (\cH (2 z + a(x_u)))  
\end{eqnarray*}

\noindent
The only equation that remains to be addressed in Theorem \ref{thm:susysolutions} to guarantee that we have a local susypersymmetric solution to NS-NS supergravity is the differential equation for $\mathfrak{K}$, which reduces simply to $\partial_z \mathfrak{H} = 0$. Its general solution is:
\begin{eqnarray*}
\cH = \frac{l(x_u)}{2 z + a(x_u)} 
\end{eqnarray*}

\noindent
for an arbitrary function $l(x_u)$ of $x_u$. Hence, we obtain the following local characterization of supersymmeric NS-NS solutions.

\begin{cor}
Let $(g,b,\phi)\in \Sol_s(\cQ,\cX)$ be a supersymmetric solution. Then, there exist local coordinates on $(x_u,x_v,z)$ in which $(g,b,\phi)$ is given by:   
\begin{align*}
& g = \frac{l(x_u)}{2 z + a(x_u)}  \dd x_u\otimes \dd x_u +  \frac{\partial_{x_u} a(x_u)}{2 z + a(x_u)}  \dd x_u\odot \dd x_v + \dd z \otimes \dd z\\
& b = \frac{2\dd x_v\wedge \dd z + \dd x_u\wedge \dd x_v}{2 z + a(x_u)}\, , \qquad \phi = - \log(2 z + a(x_u))
\end{align*}
	
\noindent
for a pair of functions $a(x_u)$ and $l(x_u)$ depending only on $x_u$.
\end{cor}

\noindent
For the previous local expressions to be well defined we must have:
\begin{eqnarray*}
2 z + a(x_u) > 0\, , \qquad \partial_{x_u} a(x_u) > 0
\end{eqnarray*}

\noindent
so the local domain of the coordinates needs to be restricted accordingly. Modulo these constraints, we obtain the following corollary.

\begin{cor}
Supersymmetric solutions of three-dimensional NS-NS supergravity are locally parametrized by two real functions of one variable. 
\end{cor}


\appendix


\phantomsection
\bibliographystyle{JHEP}

\end{document}